\documentclass[12pt]{iopart}

\usepackage{amsthm}

\usepackage{graphicx}

\usepackage{subfig}

\usepackage{url}
\usepackage{hyperref}

\usepackage{cases}

\theoremstyle{plain} 
\theoremstyle{remark} 
\theoremstyle{plain} 
\theoremstyle{plain} 
\theoremstyle{plain} 
\theoremstyle{break}

\newtheorem{theorem}{Theorem}

\newtheorem{lemma}[theorem]{Lemma}

\usepackage{iopams}
\begin{document}

\title[Simultaneous Reconstruction and Segmentation for Dynamic SPECT Imaging]{Simultaneous Reconstruction and Segmentation for Dynamic SPECT Imaging}

\author{Martin Burger$^1$, Carolin Rossmanith$^1$, Xiaoqun Zhang$^2$}
\vspace{10pt}

\address{$^1$ Institute for Computational and Applied Mathematics, University of M\"unster, Einsteinstra\ss e 62, 48149 M\"unster, Germany}
\vspace{10pt}

\address{$^2$ Institute of Natural Sciences, School of Mathematical Sciences, and MOE-LSC, Shanghai Jiao Tong University, 800 Dongchuan Road, 200240 Shanghai, China}
\vspace{10pt}

\ead{martin.burger@wwu.de,xqzhang@sjtu.edu.cn}
\vspace{10pt}

\begin{indented}
\item[]December 2015
\end{indented}

\begin{abstract}
This work deals with the reconstruction of dynamic images that incorporate characteristic dynamics in certain subregions, as arising for the kinetics of many tracers in emission
tomography (SPECT, PET). We make use of a basis function approach for the unknown tracer concentration by assuming  that the region of interest can be divided into subregions with
spatially constant concentration curves. Applying a regularized variational framework reminiscent of the Chan-Vese model for image segmentation we simultaneously reconstruct both the
labelling functions of the subregions as well as the subconcentrations within each region. Our particular focus is on applications in SPECT with  Poisson noise model, resulting in a Kullback-Leibler
data fidelity in the variational approach.

We present a detailed analysis of the proposed variational model and prove existence of minimizers as well as error estimates. The latter apply to a more general class of problems and generalize existing
results in literature since we deal with a nonlinear forward operator and a nonquadratic data fidelity. A computational algorithm based on alternating minimization and splitting techniques is developed for the
solution of the problem and tested on appropriately designed synthetic data sets. For those we compare the results to
those of standard EM reconstructions and investigate the effects of Poisson noise in the data.
\end{abstract}

%
%
%
%
%

\section{Introduction}

\textit{Single Photon Emission Computed Tomography} (SPECT) is a popular medical imaging technique which, like \textit{Positron Emission Tomography} (PET), provides an insight into
the physiological processes of a living organism. A radioactive substance, the so-called radiotracer, is injected into the patient's body and participates in its natural processes
in a way depending on the particular properties of the tracer substance. For example, a radioactive type of glucose might be appropriate to visualize active tumor cells within healthy
tissue. The emission of gamma rays is measured by surrounding gamma cameras, which contain a perforated honeycomb-like sheet, the so-called collimator, to ensure that only photons
from a specified angle are counted. The measurements in turn can be used to gain a visualization of the tracer concentration \cite{Wernick2004}.

Whereas static SPECT imaging yields a single static image neglecting any temporal aspects within a physiological process, dynamic SPECT reconstruction aims at providing a series of
images visualizing time-dependent processes within an organism. From the mathematical point of view, the connection between given measurements
$g$ and the unknown time-dependent tracer concentration $f\left(x,t\right)$ in every point $x$ of a certain region $\Omega\subset\mathbb{R}^d$ for $d=2$ or $d=3$ at time $t$ is given
by the so-called attenuated Radon transform
\begin{equation}
\label{eq:RadonOperator}
g\left(\theta,s,t\right) = \int_{L\left(\theta,s\right)} f\left(x,t\right)e^{-\int_{H\left(x\right)}\mu\left(y\right)dy} dx = \mathcal{R}f\left(\theta,s,t\right)
\end{equation}
where $L\left(\theta,s\right) := \left\{t\theta+s\theta^{\bot}:t\in\mathbb{R}\right\}$ and $s\in\mathbb{R}$ defines the position of a collimator bin while the camera position is
specified by the angle $\theta\in\left[0,2\pi\right)$. $H\left(x\right)$ specifies the section of the straight line $L$ which lies in between $x$ and the collimator and
$\mu:\Omega\rightarrow[0,\infty)$ defines the grade of attenuation of gamma rays on their way through different types of tissue. In this work, we make the assumption that $\mu$ is
known a priori and directly included in the Radon transform. In small animal SPECT, the attenuation map is negligible,
while for human SPECT scans, one usually uses a given attenuation map which is pre-estimated before, for example from an additional scan. Another approach would be to perform a
simultaneous reconstruction of both the tracer concentration and the attenuation map, but this is beyond the scope of this work, hence we can ignore it in the following
\cite{Wuebbeling2001}.

The most intuitive approach to reconstruct $f$ out of given measurements $g$ is to perform static SPECT reconstruction for each time step independently. Obviously, the main drawback
of this idea is the fact that any temporal correlation between the images is neglected. Furthermore, the approach might not provide suitable results: In order to guarantee a good
temporal resolution, we assume that every camera position is accompanied by a new time step and therefore we only have access to measurements from one (or two, depending on the actual
number of rotating cameras) angles per time step. As a result, the solution is strongly underdetermined. We therefore need to search for a reconstruction method which includes the
temporal coherence and allows the integration of given a priori information about the structure of the solution.

There exist several discussions about dynamic image reconstruction in emission tomography. An overview about the state of the art in both PET and SPECT imaging is given in
\cite{Reader2014}. The main reconstruction methods as well as several clinical applications of dynamic SPECT are listed in \cite{Gullberg2010}. The ideas are commonly based on kinetic
modeling (cp. section \ref{sec:SimultaneousSegmentationAndReconstruction}) and aim at minimizing the Kullback-Leibler divergence between the given data and the Radon transform of the
underlying tracer distribution. An interesting approach which is comparable to the one in this paper is presented in \cite{Zan2013}. Here, the authors construct a sparse binary matrix defining
the affiliation of a pixel to a certain region by pre-segmentation of a static reconstructed image, assuming that the tracer concentration within certain regions remains spatially
constant. Afterwards, the concentration in each region is modelled by a linear combination of B-spline basis functions. In \cite{Chen2015}, a simultaneous segmentation and reconstruction approach
of dynamic PET data is performed by decomposing the spatial region into foreground and background regions, which also have a spatially constant concentration curve and differ in the
magnitude of the latter.

The main contribution of this work is the development and analysis of a novel approach for dynamic SPECT reconstruction based on image segmentation. The paper is organized as
follows. In Section \ref{sec:SimultaneousSegmentationAndReconstruction}, we present the proposed simultaneous segmentation and reconstruction model.  Then in Section
\ref{sec:analysis}, the existence of minimizers and error estimates of the proposed variational model are provided. In Section \ref{sec:alg}, we present the numerical algorithm to
solve the proposed variational problems. Finally, the numerical results are presented in Section \ref{sec:num}.


\section{Simultaneous Segmentation and Reconstruction model}
\label{sec:SimultaneousSegmentationAndReconstruction}

In order to include the temporal correlation between the slices of the image sequence, there already exists a common approach which we use a basis for a slightly different model. The
approach is based on compartmental modeling of the region of interest and is described in detail in \cite{Wernick2004}. The region $\Omega$ is assumed to be separated into so-called
compartments: regions, in which the tracer concentration varies in time, but not in space. One compartment can either be one pixel (or voxel) or a whole region consisting of the same
tissue. One of the main model assumption is furthermore that the tracer input is represented by the concentration in the blood, which is modelled as a function of time and is assumed
to be known (since it can be measured separately). Hence, the blood vessel is not seen as a compartment as such, but as an additional region where the concentration is given a priori.

The compartmental modelling starts from a very simple model, which only describes the tracer flux between blood and a single tissue type described by a differential equation for the
unknown concentration in the tissue. In a subsequent approach, the blood compartment is extended by distinguishing between an arterial and a venous part. A relation between these
parts is modelled applying the Fick principle. This leads to an extended differential equation for the concentration in the tissue $C_T$
\begin{equation}
\label{eq:ODEKineticModeling}
\frac{d}{dt}C_T\left(t\right) = D\left(C_A\left(t\right)-C_V\left(t\right)\right)
\end{equation}
where $C_A,C_V$ are the concentration in the arterial or, likewise, the venous part of the blood vessel and $D$ is the blood flow rate (i.e. the volume of blood that passes through a vessel
per unit time). By the assumption that the concentration reaches an equilibrium state very fast, the rate between tissue and venous concentration $\lambda=\frac{C_T}{C_V}$ can be
regarded as a constant, such that substituting $C_V(t)$ by $\frac{C_T(t)}{\lambda}$ leads to
\begin{equation}
\frac{d}{dt}C_T\left(t\right) = D\left(C_A\left(t\right)-\frac{C_T(t)}{\lambda}\right)
\end{equation}
This model can again be extended by separating the tissue
region into several independent compartments and therefore adding a space component to the concentration in the tissue $C_T$. This way, equation (\ref{eq:ODEKineticModeling}) turns
into a partial differential equation
\begin{equation}
\frac{\partial}{\partial t}C_T\left(x,t\right) = D\left(x\right)\left(C_A\left(t\right)-\frac{C_T\left(x,t\right)}{\lambda}\right)
\end{equation}
and the solution with initial value $C_T\left(x,0\right)=0$ can easily be computed as
\begin{equation}
C_T\left(x,t\right) = D\left(x\right) \int_0^t C_A\left(\tau\right)e^{-\frac{D\left(x\right)}{\lambda}\left(t-\tau\right)}d\tau
\end{equation}

Under the assumption that we have prior knowledge about $\frac{D\left(x\right)}{\lambda}$, we can provide a set of possible values $\widetilde{c_k}$, $k=1,\ldots,K$ for some fixed
number $K\in\mathbb{N}$ for this factor and therefore, obtain a set of possible concentration curves. This set might not contain the true concentration curve in every single
compartment, but still it makes sense to conclude that every concentration curve can be written as a linear combination of at most a few of these basis functions. This leads to the
common approach
\begin{equation}
f\left(x,t\right) = \sum_{k=1}^K u_k\left(x\right) \underbrace{\int_0^t C_A\left(\tau\right)e^{-\widetilde{c_k}\left(t-\tau\right)}d\tau}_{=c_k\left(t\right)}
\end{equation}
where for every $x$, only a few coefficients $u_k\left(x\right)$ are nonzero, which leads to a sparsity term in our model presented within this section.

The idea of kinetic modelling (see \cite{Wernick2004}) is usually applied for dynamic PET reconstruction as in \cite{Benning}, but also for SPECT imaging
\cite{Huesmann1998}, \cite{Maeght}. Another comparable approach uses splines as temporal basis functions \cite{Asma2006} or \cite{ding2015dynamic}. In \cite{Reader2006}, the authors provide an algorithm which
alternatingly estimates the coefficients and updates the temporal basis, where only the number of basis function must be specified, whereas \cite{Zan2013} uses a pre-segmentation of
the image region and B-spline basis functions for the temporal tracer concentration curves. In \cite{ding2015dynamic}, an alternating updating of low rank matrix factorization model
is proposed and further constraints are introduced to limit the solution space. Another approach already mentioned in the introduction which is also based on low-rank matrix
decomposition is presented in \cite{Chen2015}.

The new approach which is introduced here is, referring to the underlying philosophy, a slightly different one, but will lead to a similar model. We assume that the contemplated
region $\Omega$ of the patient's body can be separated into a certain number of disjoint regions $\Omega_1,\ldots,\Omega_K$, whose boundaries remain static over time. The tracer
distribution in every region now does not spatially change, so that every $\Omega_k$ has its own space-independent concentration curve $c_k\left(t\right)$.

In the mathematical sense, this approach leads to a very similar model to the one emerging from the basis pursuit. Under the assumption that every spatial point $x$ belongs to
exactly one region, $f$ can be written as a sum of the regional concentrations $C_k\left(t\right)$ and spatial labelling functions $u_k\left(x\right)$, i.e.
\begin{eqnarray}
f\left(x,t\right) &= \sum_{k=1}^K u_k\left(x\right)c_k\left(t\right)
= \left(\begin{array}{ccc} u_1\left(x\right) & \ldots & u_K\left(x\right) \end{array}\right)
  \left(\begin{array}{c} c_1\left(t\right) \\ \vdots \\ c_K\left(t\right) \end{array}\right) \nonumber \\
&=: u(x)c(t)
\end{eqnarray}
where
\begin{eqnarray}
u_k\left(x\right) = \left\{ \begin{array}{cl} 1 & \mbox{if } x\in\Omega_k \\ 0 & \mbox{otherwise} \end{array} \right.
\end{eqnarray}
In practice, this model makes sense and is easily transferable to the anatomical reality, e.g. thinking of roughly separating the patient's body into different tissue types, where
each has its own unique chemical texture, which in turn causes a different behaviour of the added radiotracer.

The task arising from this model is to reconstruct the indicator functions $u_k$ and to search for the subregional concentration curves $c_k$. Thereby, a reconstruction of the
radiotracer concentration and a simultaneous segmentation of the region of interest are performed.

The solution is formulated as a minimizer of the general variational framework
\begin{eqnarray}
\label{eq:ContinuousVarModel}
\min_{u\in S, \ c\geq 0} &\ KL\left(\mathcal{R}uc,g\right)+\alpha\sum_{k=1}^K\left| u_k\right|_{BV\left(\Omega\right)}+\beta\sum_{k=1}^K\left\|u_k\right\|_
{L^1\left(\Omega\right)} \nonumber \\
&+\frac{\delta}{2}\sum_{k=1}^K\left\|\frac{\partial}{\partial t}c_k\right\|^2_{L^2\left(\left(0,T\right)\right)}
\end{eqnarray}
where
\begin{eqnarray}
\label{eq:Sdefinition}
S:=\left\{v: \sum_{k=1}^Kv_k\left(x\right)=1 \ \forall x, \ v_k\left(x\right)\in\left[0,1\right] \ \forall x\right\}
\end{eqnarray}
and $KL$ represents the Kullback Leibler divergence
\begin{eqnarray}
KL\left(\mathcal{R}uc,g\right) := \int_{\Sigma} \int_0^T \left(\mathcal{R}uc-g+g\log\left(\frac{g}{\mathcal{R}uc}\right)\right)
\end{eqnarray}
which is the data fidelity of choice in case of Poisson noise-corrupted data. $\mathcal{R}:\Omega\times\left[0,T\right]\rightarrow\Sigma\times\left[0,T\right]$ is the Radon operator
as mentioned in (\ref{eq:RadonOperator}), which transforms an image sequence $f$ into sinogram data $g$. As regularization terms we chose for the indicator functions $u$ the total
variation for each subregion to enforce sharp edges as well as a convex sparsity regularization via the $L^1$-norm. Furthermore, we include a smoothness regularization for the tracer
concentration in each subregion by penalizing the $L^2$-norm of the gradient.


\section{Analysis of the Variational Problem}
\label{sec:analysis}

We want to investigate the variational problem of minimizing the functional

\begin{eqnarray}
 J:BV\left(\Omega\right)^K\times W_0^{1,2}\left(\left(0,T\right)\right)^K\rightarrow\mathbb{R}\cup\left\{+\infty\right\}
\end{eqnarray}
defined via
\begin{eqnarray}
J\left(u,c\right) = &KL\left(\mathcal{R}uc,g\right)+\alpha\sum_{k=1}^K\left| u_k\right|_{BV\left(\Omega\right)}+\beta\sum_{k=1}^K\left\|u_k\right\|_{L^1\left(\Omega\right)} \nonumber \\
&+ \delta/2\sum_{k=1}^K\left\|\frac{\partial c_k}{\partial t}\right\|^2_{L^2\left(\left(0,T\right)\right)}
\end{eqnarray}
where $\alpha,\beta,\delta\geq 0$ are regularization parameters which have to be chosen in a suitable way. Furthermore, we assume $\log(g)\in L^1(\Sigma\times[0,T])$. We start with
the basic question of existence of minimizers and then proceed to stability estimates, which are carried out in terms of Bregman distances as nowadays common for
nonlinear variational regularization (cf. \cite{Burger2004,burger2015bregman}).
%

\subsection{Existence of a Minimizer}

In order to prove existence of a minimizer we employ standard compactness and lower semicontinuity arguments. Indeed we shall look for a minimizer satisfying the additional
constraints
\begin{eqnarray}
\label{eq:constraints}
u \in S, c \geq 0.
\end{eqnarray}

Let us first clarify the topologies to work with. As we shall see below the functional $J$ is indeed coercive in the norms of
$BV\left(\Omega\right)^K\times W_0^{1,2}\left(\left(0,T\right)\right)^K$, and we can thus obtain the Banach-Alaoglu theorem to immediately obtain weak-* respectively weak star
compactness. Note for this sake that $BV(\Omega)$ is the dual of a Banach space (cf. \cite{burger2008}) and $W^{1,2}\left(\left(0,T\right)\right)$ is a Hilbert space.

\begin{lemma}[Compactness of sub-level sets]~ \\
Let $\alpha > 0$, $\beta \geq 0$ and $\delta > 0$.
Then there exists $a > 0$ such that
$$
{\cal S}_a = \left\{\left(u,c\right)\in S \times W_0^{1,2}\left(\left(0,T\right)\right)^K: \ J\left(u,c\right)\leq a\right\}
$$
is not empty and compact in the weak-* topology.
\end{lemma}

\begin{proof}
First of all $\tilde c_k(t) \equiv t$ and $\tilde u_k(x) \equiv \frac{1}K$ are admissible elements that lead to a finite value of $J$, hence choosing $a = J(\tilde u,\tilde c)$ yields
a nonempty ${\cal S}_a$. Now let $(u,c) \in {\cal S}_a$, then the nonnegativity of each term yields
$$
\left|u_k\right|_{BV\left(\Omega\right)} \leq \frac{a}{\alpha}, \qquad k=1,\ldots,K
$$
and the fact that $u \in S$ has bounded components in $L^\infty(\Omega)$ yields a bound for the norm of $u$ in $BV(\Omega)^K$.

For $c$, we find by the same arguments
$$
\left\|\frac{\partial}{\partial t}c_k\right\|_{L^2\left(\left(0,T\right)\right)} \leq \sqrt{\frac{2a}{\delta}}, \qquad k=1,\ldots K,
$$
and the left-hand side is an equivalent norm on $W^{1,2}_0\left(\left(0,T\right)\right)$ due to the Poincar\'{e}-Friedrichs inequality. Finally the Banach-Alaoglu theorem yields the
assertion.

\end{proof}

Secondly, we show the lower semi-continuity of the functional $J$: \\

\begin{lemma}[Lower semi-continuity]~ \\
Let $\alpha, \beta, \delta \geq 0$, $u\in BV\left(\Omega\right)$ and $c\in W^{1,2}\left(\left(0,T\right)\right)$. Then the functional $J$ is lower semi-continuous
on the constraint set with respect to the weak-* topology in $BV\left(\Omega\right)^K$ and weak topology in $W^{1,2}\left(\left(0,T\right)\right)$.
\end{lemma}

\begin{proof}
We can verify the lower semicontinuity of all terms separately due to the additive structure of $J$. Since the norm in $BV(\Omega)$ is the convex conjugate of the characteristic
function of the unit ball in its predual space, it is lower semicontinuous in the weak-star topology (cf. also \cite{burger2008} for an explicit proof). Due to the compact
embedding of $BV(\Omega)$ into $L^1(\Omega)$, the $L^1$-norms are even continuous. The squared equivalent norms in $W_0^{1,2}((0,T))$ are lower semicontinuous due to weak
semicontinuity of norms in reflexive Banach spaces (\cite{Ekeland1987}). It hence remains to verify the lower semicontinuity of the data fidelity term. For this sake we employ the
weak sequential lower semicontinuity of the Kullback-Leibler divergence in $L^1(\Omega \times (0,T))$ (cf. \cite{Resmerita2007}, Lemma 3.4). Consequently it remains to show that for nonnegative
sequences $u_k^j \rightharpoonup^* u_k$ in $BV(\Omega)$ and $c_k^j \rightharpoonup c_k$ in $W_0^{1,2}((0,T))$ we obtain weak $L^1$-convergence of $f^j= \sum_{k=1}^K u_k^j c_k^j$.
However, from the compact embeddings into $L^1(\Omega)$ and $L^1((0,T))$ we obtain an even stronger property, namely strong $L^1$-convergence of $f^j$ due to
\begin{eqnarray*}
\fl &\int_{\Omega}\int_0^T |f^j(x,t) - f(x,t)|~dt~dx \\
\fl = &\int_{\Omega}\int_0^T \left|\sum_{k=1}^K( u_k^j\left(x\right)c_k^j\left(t\right)-u_k^j\left(x\right)c_k\left(t\right)+u_k^j\left(x\right)c_k\left(t\right)-u_k\left(x\right)c_k\left(t\right))\right| dtdx \\
\fl \leq &\sum_{k=1}^K \int_{\Omega}\int_0^T \left|u_k^j\left(x\right)c_k^j\left(t\right)-u_k^j\left(x\right)c_k\left(t\right)\right|+\left|u_k^j\left(x\right)c_k\left(t\right)-u_k\left(x\right)c_k\left(t\right)\right| dtdx \\
\fl \leq &\sum_{k=1}^K \int_{\Omega} \left|u_k^j\left(x\right)\right| dx \underbrace{\int_0^T \left|c_k^j\left(t\right)-c_k\left(t\right)\right| dt}_{\rightarrow \ 0}  +
 \sum_{k=1}^K \int_0^T \left|c_k\left(t\right)\right|dt \underbrace{\int_{\Omega} \left|u_k^j\left(x\right)-u_k\left(x\right)\right| dx}_{\rightarrow \ 0} .
\end{eqnarray*}

\end{proof}

Finally we need to prove the closedness of the constraint set to finish the proof of existence. Note that we have a compact embedding of $BV(\Omega)$ into $L^1(\Omega)$ and of
$W_0^{1,2}((0,T))$ into $L^2((0,T))$. Those are sufficient to pass to the limit in lower and upper bounds on $u_k$ respectively $c_k$. Moreover, interpreting $1$ as an $L^1$ function,
by the equation
$$
\sum_{k=1}^K u_k = 1.
$$
$\sum_{k=1}^K u_k$ is continuous in $L^1$ and hence the constraint set is closed. Summing up, we obtain the following existence result for $J$:

\begin{theorem}
Let $\alpha >0$, $\beta \geq 0$, $\delta > 0$. Then there exists a minimizer of $J$ in $S \times W^{1,2}_0((0,T))$ with all $c_k$ nonnegative almost everywhere.
\end{theorem}


\subsection{Error estimates}

In the following we briefly want to outline a popular application of the Bregman distance as stated in \cite{Burger2004} in order to derive error estimates for minimizers of
regularized variational models in the context of inverse problems. Although this result can state the convergence under certain conditions, it is based on a nonlinearity condition
(\ref{eq:NonlinearityCondition}), which we cannot verify for our particular case. Since other convergence theorems which could be used here also suffer from slightly different
conditions on the data (see for example \cite{Werner2015},\cite{Werner2012}), we nevertheless state the results here. For more details, see \cite{Rossmanith2014}.

Let $\mathcal{X}$ and $\mathcal{Y}$ be some Banach spaces. To keep the notation simple we will use the notation
$v=(u,c)\in\mathcal{X}$ throughout this part. Let $\widetilde{v}\in\mathcal{X}$ be the exact solution
of the inverse problem $G(v)=g$ with a continuous and  Fr\'{e}chet-differentiable nonlinear operator $G:\mathcal{X}\rightarrow\mathcal{Y}$. In practice we face the problem that we only have access to noisy data
$g^{\delta}$. We are interested in an approximate solution of $G\left(v\right)=g^{\delta}$ which is close to the exact solution $\widetilde{v}$. Therefore we
apply a regularized variational technique as above, which in short-hand notation is written as
\begin{eqnarray}
\label{eq:GeneralVarModel}
\min_v \ F_{g^{\delta}}\left(G\left(v\right)\right)+\alpha R\left(v\right)
\end{eqnarray}
with a data fidelity $F_{g^{\delta}}:\mathcal{Y}\rightarrow\mathbb{R}\cup\left\{+\infty\right\}$ and convex regularization term $R:\mathcal{X}\rightarrow\mathbb{R}\cup\left\{+\infty\right\}$
with parameter $\alpha\geq0$. We assume in the abstract setting that $F_{g^{\delta}}\circ G$ is Fr\'{e}chet-differentiable, which is indeed true in our setting. In order to estimate
errors between a minimizer $\hat{v}$ of (\ref{eq:GeneralVarModel}) and $\tilde{v}$ we shall use Bregman distances and, as well-known in regularization theory, need some kind of
source condition.

As in \cite{Resmerita2006} we use the source condition
\begin{eqnarray}
\label{eq:SourceCondition}
\exists \ \rho\in\partial R\left(\widetilde{v}\right), \ \exists \ q\in\mathcal{Y}^*\setminus\left\{0\right\}: \ \rho=G'\left(\widetilde{v}\right)^*q
\end{eqnarray}

Here, $\partial R$ denotes the subgradient of $R$, $G'$ the Fr\'{e}chet-derivative of the operator $G$ and $(G')^*$ its convex conjugate. Combining ideas from
\cite{benning2011error} for non-quadratic fidelities and \cite{Resmerita2006} for nonlinear forward operators we obtain the following result: \\

\begin{theorem}[Error estimate in the Bregman distance]~ \\
\label{th:ErrorEstimateBregman}
Let $\widetilde{v}$ be the exact solution of the inverse problem $G\left(v\right)=g$ and let the source condition (\ref{eq:SourceCondition}) be satisfied. Let
$R:\mathcal{X}\rightarrow\mathbb{R}\cup\left\{+\infty\right\}$ be convex. Furthermore, let the nonlinearity condition
\begin{eqnarray}
\label{eq:NonlinearityCondition}
\langle q,G\left(v\right)-G\left(w\right)-G'\left(w\right)\left(v-w\right)\rangle \leq C\left\|q\right\|_{\mathcal{Y}^*}\left\|G\left(v\right)-G\left(w\right)\right\|_{\mathcal{Y}}
\end{eqnarray}
hold for $q$ from equation (\ref{eq:SourceCondition}). If there exists a minimizer $\hat{v}$ of the variational model (\ref{eq:GeneralVarModel}) for $\alpha>0$ which satisfies the KKT optimality conditions, then
$$
F_{g^{\delta}}\left(G\left(\hat{v}\right)\right) + \alpha D_R^{\rho}\left(\hat{v},\widetilde{v}\right) \leq F_{g^{\delta}}\left(g\right)+\alpha\left(C+1\right)
\left\|q\right\|_{\mathcal{Y}^*}\left\|G\left(\hat{v}\right)-G\left(\widetilde{v}\right)\right\|_{\mathcal{Y}}
$$

\end{theorem}

\begin{proof}
From the definition of $\hat{v}$ we obtain that $\hat{v}\in\arg\min_u \ F_{g^{\delta}}\left(G\left(v\right)\right)+\alpha R\left(v\right)$, hence it follows that
$$
F_{g^{\delta}}\left(G\left(\hat{v}\right)\right) + \alpha R\left(\hat{v}\right) \leq F_{g^{\delta}}\left(G\left(\widetilde{v}\right)\right) + \alpha R\left(\widetilde{v}\right)
$$
and by adding $-\alpha\langle \rho,\hat{v}-\widetilde{v}\rangle_{\mathcal{X}}-\alpha R(\widetilde{v})$ to both sides and $G\left(\widetilde{v}\right)=g$ we obtain
$$
F_{g^{\delta}}\left(G\left(\hat{v}\right)\right) + \alpha\underbrace{\left(R\left(\hat{v}\right)-R\left(\widetilde{v}\right)-\langle\rho,\hat{v}-\widetilde{v}\rangle_
{\mathcal{X}}\right)}_{=D_R^{\rho}\left(\hat{v},\widetilde{v}\right)} \leq F_{g^{\delta}}\left(g\right)-\alpha\langle\rho,\hat{v}-\widetilde{v}\rangle_{\mathcal{X}}
$$
By inserting the source condition $\rho=G'\left(\widetilde{v}\right)^*q$ it follows that
\begin{eqnarray*}
&F_{g^{\delta}}\left(G\left(\hat{v}\right)\right)+\alpha D_R^{\rho}\left(\hat{v},\widetilde{v}\right) \\
\leq &F_{g^{\delta}}\left(g\right)-\alpha\langle G'\left(\widetilde{u}\right)^*q,\hat{v}-\widetilde{v}\rangle_{\mathcal{X}} \\
= &F_{g^{\delta}}\left(g\right)+\alpha\langle q,-G'\left(\widetilde{u}\right)\left(\hat{v}-\widetilde{v}\right)\rangle_{\mathcal{Y}} \\
= &F_{g^{\delta}}\left(g\right)+\alpha \langle q,G\left(\hat{v}\right)-G\left(\widetilde{v}\right)-G'\left(\widetilde{v}\right)\left(\hat{v}-\widetilde{v}\right)\rangle_
{\mathcal{Y}}
 -\alpha\langle q,G\left(\hat{v}\right)-G\left(\widetilde{v}\right)\rangle_{\mathcal{Y}} \\
\leq &F_{g^{\delta}}\left(g\right) + \alpha\left(C+1\right)\left\|q\right\|_{\mathcal{Y}^*}\left\|G\left(\hat{v}\right)-G\left(\widetilde{v}\right)\right\|_{\mathcal{Y}}
\end{eqnarray*}
\end{proof}

\section{Algorithm}
\label{sec:alg}

In order to find a minimizer of the proposed variational model (\ref{eq:ContinuousVarModel}), we apply an alternating updating structure based on a forward-backward splitting EM-type
method with a variable damping parameter \cite{Sawatzky2008}. Let $n=n_1\cdot n_2$ be the total number of image pixels, $M$ the number of time steps and $K$ the number of image regions
as before. Then we can introduce a $n\times K$-matrix $U$ and a $M\times K$-matrix $C$, which correspond to the previously defined $u_1,\ldots,u_K$,
respectively $c_1,\ldots,c_K$. Each row of $U$ corresponds to one image pixel, where the entry in column $j$ states the rate of affiliation of this image pixel to the $j$-th region. Note
that due to convexity reasons, we do not force $U_{ij}\in\{0,1\}$, but allow $U_{ij}\in[0,1]$ with the restriction $\sum_{j=1}^KU_{ij}=1$ for every row $i$. The $k$-th column of the
matrix $C$ represents the tracer concentration over time within region $k$. It follows that $f=UC^T$ is a $n\times M$-matrix, where $f_{ij}$ is the tracer concentration in pixel $i$
at time $j$. Setting $m$ as the total number of detector bins (accumulated over all measurement angles), the data can be represented by a $m\times M$-matrix, where every row contains
the measurements in one detector bin over time. The Radon transform operator can be written as a $m\times n$-matrix, which can be applied to every column of $f$ to obtain the Radon
transform of each frame (cp. \cite{Rossmanith2014}). In this discrete setting, the convex set $S$ reads
\begin{equation}
S = \{ U\in\mathbb{R}^{n\times K} \ | \ \sum_{j=1}^KU_{ij}=1 \ \forall \ i, \ U_{ij}\in[0,1] \ \forall \ i,j \}.
\end{equation}

We use an unconstrained discretized model version of (\ref{eq:ContinuousVarModel})
\begin{eqnarray}
\label{eq:DiscreteVarModel}
\fl \min_{U,C^T} KL\left(\mathcal{R}UC^T,g\right)+\alpha\sum_{k=1}^K\left\|\nabla U_k\right\|_1+\beta\left\|U\right\|_1+\delta_S\left(U\right)+\frac{\delta}{2}
\sum_{k=1}^K\left\|\nabla_tC_k\right\|^2_2+\delta_+\left(C^T\right)
\end{eqnarray}
where $\delta_S$ (or likewise $\delta_+$) is given by
\begin{eqnarray}
\delta_S\left(U\right) = \left\{ \begin{array}{cl} 0 & \mbox{if } U\in S \\ \infty & \mbox{otherwise} \end{array} \right.
\end{eqnarray}

In order to compute the derivative of $\mathcal{R}UC^T$ with respect to $U$ and $C$ easily, we define linear operators $A$ and $B$ such that $A vec(U) = vec(\mathcal{R}UC^T)$ and
$B vec(C^T) = vec(\mathcal{R}UC^T)$, where $vec(\cdot)$ denotes the column-by-column vectorization of a matrix. The detailed construction of $A$ and $B$ as well as the discretization
of the operator $\mathcal{R}$ can be found in \cite{Rossmanith2014}. In the following algorithm, we use the vectorized versions of all matrices, although we still use the same notation as
before instead of writing $vec(\cdot)$ for the sake of simplicity.  Let $k$ be the outer iteration index, $A_k$ and $B_k$ be the operators such that $\ A_kvec(U):=vec(\mathcal{R}UC_k^T)$ and  $ B_{k+1}vec(C^T): =vec(\mathcal{R}U_{k+1}C^T)$. Let $U_{,\  0}:=U_k$ and $C_{,\  0}^T:=C_k^T$,  thus the subproblems for $U$ and $C^T$ are solved by the following iterations:
\begin{subequations}\label{eq:U}
\begin{numcases}{}
\widetilde{U}_{,\  i+\frac{1}{2}} = w_i (\frac{U_{,\  i}}{A_k^T1} A_k^T\left(\frac{g}{A_kU_{,\  i}}\right)) + \left(1-w_i\right) U_{,\  i} \label{eq:EMu}\\
U_{,\  i+1} = \arg\min_U \ \frac{1}{2} \left\|\frac{\sqrt{A_k^T1}\left(U-\widetilde{U}_{,\  i+\frac{1}{2}}\right)}{\sqrt{U_{,\  i}}}\right\|_2^2 + w_i\alpha\sum_{l=1}^K
\left\|\nabla_x U(:,\ l)\right\|_1 \nonumber \\
\quad\quad+ w_i\beta\sum_{l=1}^K \left\|U(:,l)\right\|_1 + w_i\delta_S\left(U\right)
\label{eq:SubproblemU}
\end{numcases}
\end{subequations}
and for $C^T$ respectively
\begin{subequations}\label{eq:C}
\begin{numcases}{}
\widetilde{C}^T_{,\ i+\frac{1}{2}} = w_i (\frac{C^T_{,\  i}}{B_{k+1}^T1} B_{k+1}^T\left(\frac{g}{B_{k+1}C^T_{,\  i}}\right)) + \left(1-w_i\right) C^T_{,\ i} \\
C^T_{,\  i+1} = \arg\min_{C^ T} \ \frac{1}{2} \left\|\frac{\sqrt{B_{k+1}^T1}\left(C^T-\widetilde{C}^T_{,\  i+\frac{1}{2}}\right)}{\sqrt{C^T_{,\ i}}}\right\|_2^2 +
w_i\frac{\delta}{2}\sum_{l=1}^K \left\|\nabla_tC(:,l)\right\|^2_2 \nonumber \\
\quad\quad+ w_i\delta_+\left(C^T\right)
\label{eq:SubproblemC}
\end{numcases}
\end{subequations}
with a damping parameter $w_i\in(0,1]$ (cp. \cite{Sawatzky2008}, \cite{burger2008}). $U(:, l)$ denotes the $l$-th column of the orginial matrix $U$, hence, in the vectorized version the entries
$(l-1)n+1, \ldots, ln$. We note that the divisions in the two steps are all performed componentwisely. In the first step of both updating procedures, a standard EM step is performed, while in the second step we have to solve a slightly simpler minimization
problem consisting of a weighted Gaussian data term plus the original regularization terms, which are now also scaled by the damping parameter $w_k$. The subproblems
(\ref{eq:SubproblemU}) and (\ref{eq:SubproblemC}) are solved via a modified primal dual hybrid gradient algorithm (\cite{Chambolle2010}\cite{esser2010general}), which in turn leads
to a set of simple minimization problems that can be solved by direct updates. The method proposes an updating scheme for functionals of the general form
\begin{eqnarray}
\arg\min_{x\in X} \ F\left(Kx\right)+G\left(x\right)
\end{eqnarray}
with finite dimensional real vector spaces $X$ and $Y$, a continuous linear operator $K:X\rightarrow Y$ and $F:Y\rightarrow [0,\infty)$, $G:X\rightarrow [0,\infty)$ convex and lower
semicontinuous. For subproblem (\ref{eq:SubproblemU}) we define
\begin{eqnarray}
F_1\left(\nabla_xU^{(1)},U^{(2)},U^{(3)}\right) &:= \omega_k\alpha\sum_l\left\|\nabla_x U^{(1)} (:, l)\right\|_1+\omega_k\beta\left\|U^{(2)}\right\|_1 \nonumber \\
                                                & \ +\omega_k\delta_S\left(U^{(3)}\right) \\
G_1\left(U\right)  &:= \frac{1}{2}\int_{\Omega} \frac{A_k^T1\left(U-\widetilde{U}_{,\ i+\frac{1}{2}}\right)^2}{U_{,\  i}} \\
K_1U 				       &:= \left(\nabla_x U,U,U\right)^T
\end{eqnarray}
and, in the same way, for subproblem (\ref{eq:SubproblemC})
\begin{eqnarray}
F_2\left(\nabla_tC^{(1)T},C^{(2)T}\right) &:= \omega_i\frac{\delta}{2}\sum_l\left\|\nabla_tC^{(1)}(:,l)\right\|^2_2+\omega_i\delta_+\left(C^{(2)T}\right) \\
G_2\left(C^T\right)  &:= \frac{1}{2}\int_{\Omega} \frac{B_{k+1}^T1\left(C^T-\widetilde{C}^T_{,\ i+\frac{1}{2}}\right)^2}{C^T_{,\  i}} \\
K_2C^T				      &:= \left(\nabla_t C^T,C^T\right)^T
\end{eqnarray}

Let us have a closer look at subproblem (\ref{eq:SubproblemU}). For $U$, as proposed in \cite{Chambolle2010}, for a parameter $\theta>0$ and steps $\sigma,\tau>0$ we obtain the
three-step method

\begin{subequations}
\begin{numcases}{}
y_{i+1} = \mbox{prox}_{\sigma F_1^*}\left( y_i+\sigma K_1v_i\right) \\
v_{i+1} = \mbox{prox}_{\tau G_1}\left(v_i-\tau K_1^*y_{i+1}\right) \\
U_{i+1} = v_{i+1}+\theta\left(v_{i+1}-v_i\right)
\end{numcases}
\end{subequations}
where the proximity operator is given by
\begin{eqnarray}
\mbox{prox}_F\left(x\right) &:= \arg\min_y \ F\left(y\right)+\frac{1}{2}\left\|y-x\right\|_2^2
\end{eqnarray}

This method mainly consists of two minimization problems for $F^*$ and $G$ respectively, and therefore splits the original functional into two separated parts. Each minimization
problem can be interpreted as an implicit gradient descent method. To avoid the computation of the conjugate functional $F^*$, we can apply the famous Moreau's identity
\cite{moreau1965}:

\begin{equation}
x = \mbox{prox}_{\tau f}\left(x\right)+\tau \mbox{prox}_{\frac{1}{\tau}f^{\ast}}\left(\frac{x}{\tau}\right)
\end{equation}

Then, the first step changes to
\begin{equation}
y_{i+1}= z_i - \sigma \ \arg\min_y \ \frac{1}{\sigma} F_1\left(y\right)+\frac{1}{2}\left\|y-\frac{1}{\sigma}z_i\right\|^2_2
\end{equation}
where $z_i = y_i+\sigma K_1v_i$. From the definition of the operator $K_1$ we see that the dual iterate $y_i$ is actually three-dimensional and most parts of $F_1$ contain only one component. Therefore, we can
separate the single components to receive the minimization problems
\begin{subequations}
\begin{numcases}{}
y^1_{i+1} = z^1_i-\sigma \ \arg\min_{y^1} \ \frac{1}{\sigma}w_k\alpha\sum_i\left\|y_{:,i}^1\right\|_1+\frac{1}{2}\left\|y^1-\frac{1}{\sigma}z_i^1\right\|_2^2 \\
y^2_{i+1} = z^2_i-\sigma \ \arg\min_{y^2} \ \frac{1}{\sigma}w_k\beta\left\|y^2\right\|_1+\frac{1}{2}\left\|y^2-\frac{1}{\sigma}z_i^2\right\|_2^2 \\
y^3_{i+1} = z^3_i-\sigma \ \arg\min_{y^3} \ \frac{1}{\sigma}w_k\delta_S\left(y^3\right)+\frac{1}{2}\left\|y^3-\frac{1}{\sigma}z_i^3\right\|_2^2.
\end{numcases}
\end{subequations}

The iterates $y^1_{i+1}$ and $y^2_{i+1}$ can be directly computed via the well-known soft shrinkage operator. $y^3_{i+1}$ is simply the projection of the minimizer of the $L^2$-term
onto the convex set $S$. The remaining minimization problem for $v_{i+1}$ only consists of two $L^2$-norm terms and can be solved directly as well. In the same way we can proceed
for subproblem (\ref{eq:SubproblemC}). In this case, the resulting minimization problems are quite similar to the ones above and can be solved without more effort.


\section{Numerical Results}
\label{sec:num}

\subsection{Synthesized data tests}

To test our reconstruction method, we created different synthesized data sets consisting of $64\times 64$ pixels with varying complexity of the image structure and with a different
number of subregions and a total of 90 time steps in each sequence. Based on these subregions, we created three-dimensional matrices containing the true image in one time step
within each slice. The underlying concentration in each subregion is related to a realistic shape of a the time-dependent behaviour of the tracer in different tissue types.

The first data set consists of a heart-shaped region and three circles on a static background (see figure \ref{fig:Phantoms} (a)). The two smaller circles are assumed to belong to
the same tissue type and therefore to the same subregion, which causes a total of four subregions, including the background. To simulate a more realistic application of dynamic SPECT
imaging, we used a synthesized representation of a rat liver as a second data set (see figure \ref{fig:Phantoms} (b)). The temporal concentration curves used to simulate the data sets
are shown in figure \ref{fig:TimeCurves}. As before, the total number of subregions was chosen to be equal to four in order to provide a both simple and realistic shape model. \\

\begin{figure}[ht]
\centering
\subfloat[First data set]{\includegraphics[width=7cm]{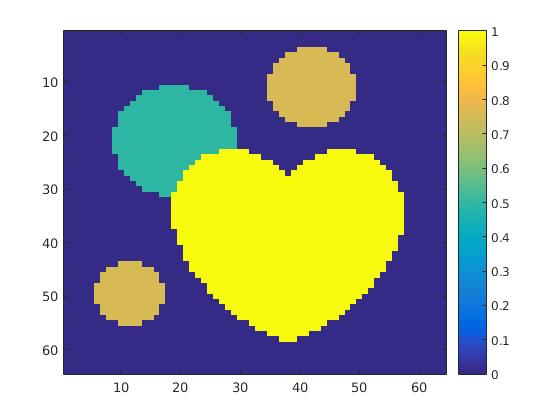}}\qquad
\subfloat[Second data set]{\includegraphics[width=7cm]{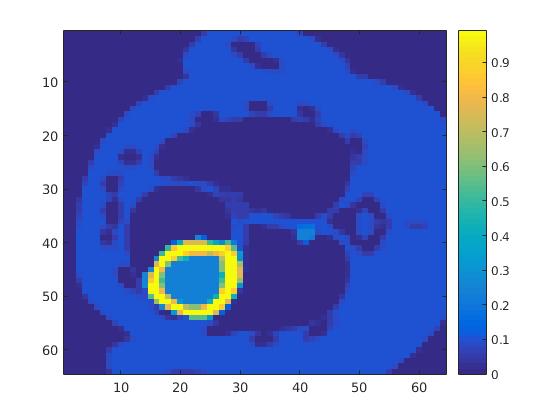}}\qquad
\caption{Phantoms of the tested data sets}
\label{fig:Phantoms}
\end{figure}

\begin{figure}[ht]
\centering
\subfloat[First data set]{\includegraphics[width=7cm]{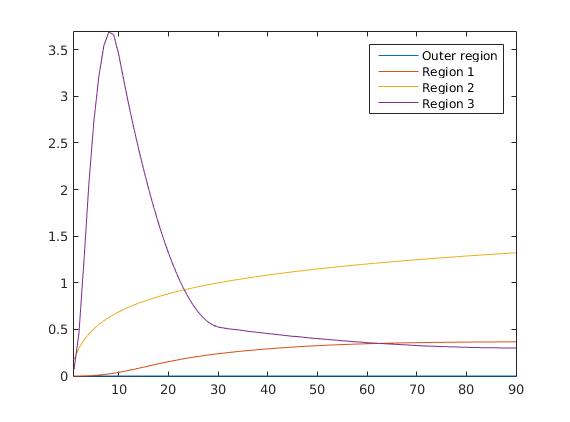}}\qquad
\subfloat[Second data set]{\includegraphics[width=7cm]{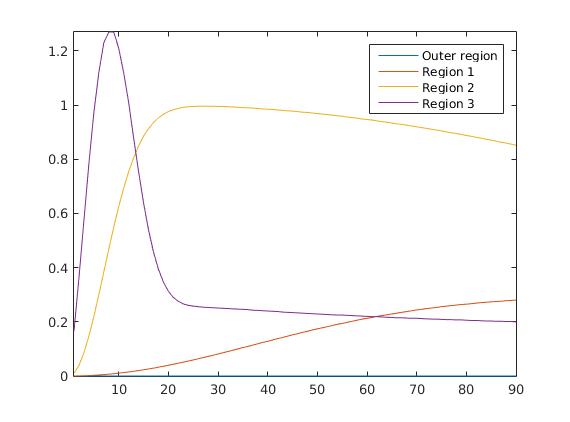}}\qquad
\caption{Temporal concentration curves in all subregions. (a) Region 1 corresponds to the circle which is partly behind the heart, region 2 are the two outer circles and region 3 is
the heart. (b) Region 1 represents the outer tissue in the phantom, region 2 the ring-shaped region of the liver and region 3 the inner of region 2.}
\label{fig:TimeCurves}
\end{figure}

To simulate the synthesized SPECT data, we apply a Radon transform assuming a double detector gamma camera, which counts photons from two opposing projection angles per
time step. For the more simple data set, we let the camera rotate clockwise around two degree per time step, in case of the complex data set we used modified projection angles, i.e.
the camera alternatingly projects from an angle of i and 45+i degrees, in order to simplify the reconstruction. Each collimator consists of $95$ detector bins, so we obtain $190$ data
points per time step and projection angle. The resulting sinogram data of the two underlying data sets are shown in figure \ref{fig:Sinograms}.

\begin{figure}[ht]
\centering
\subfloat[Sinogram data of first data set]{\includegraphics[width=7cm]{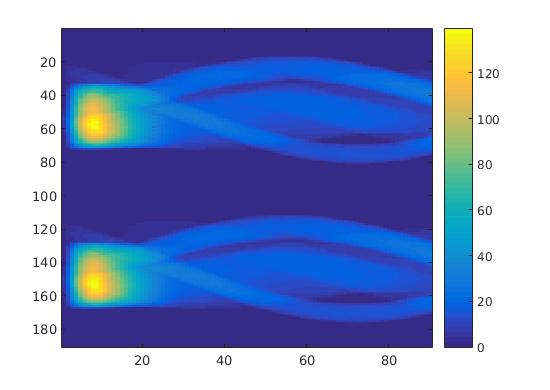}}\qquad
\subfloat[Sinogram data of second data set]{\includegraphics[width=7cm]{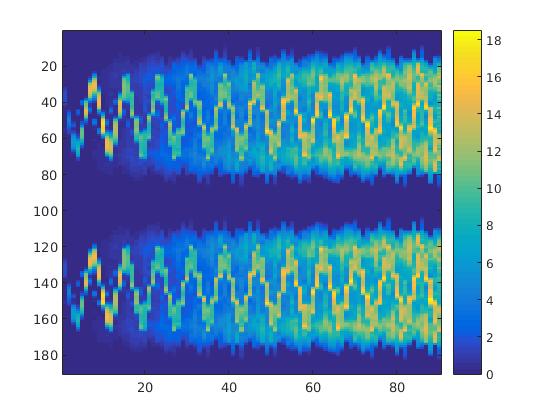}}\qquad
\caption{Sinograms of the tested data sets. For the first one, consecutive camera angles were used, for the second one, the angles alternatingly equal i and 45+i degrees.}
\label{fig:Sinograms}
\end{figure}

\begin{figure}[ht]
\centering
\subfloat[Sinogram data of first data set with Poisson noise]{\includegraphics[width=7cm]{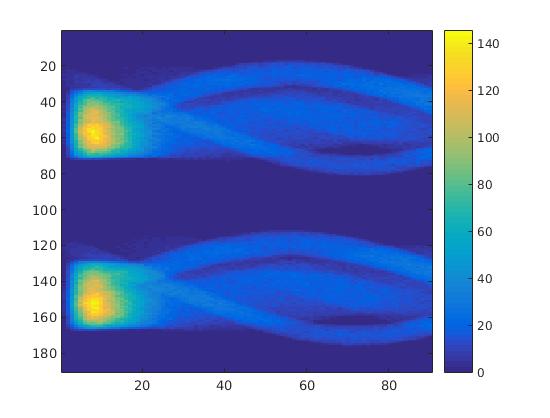}}\qquad
\subfloat[Sinogram data of second data set with Poisson noise]{\includegraphics[width=7cm]{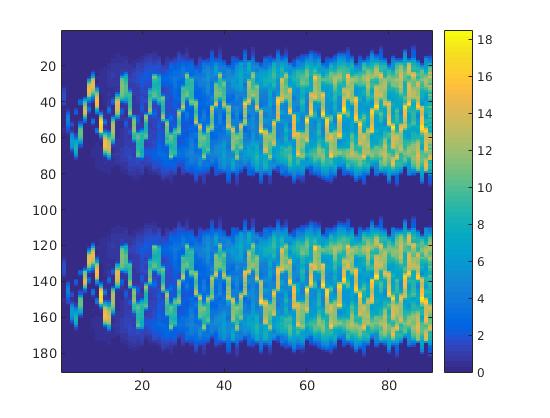}}\qquad
\caption{Sinograms of the tested data sets corrupted with Poisson noise (cp. figure \ref{fig:Sinograms}).}
\label{fig:NoisySinograms}
\end{figure}

Both data sets were reconstructed via the previously described forward-backward EM-type method in MATLAB$^{\copyright}$. $\sigma$ and $\tau$ in the subproblems are chosen as one over
the maximum over the sums of all rows (resp. columns) of the corresponding operator $K$ to guarantee the condition $\sigma\tau\|K\|^2<1$ for which convergence was shown in
\cite{Chambolle2010}. According to \cite{Chambolle2010}, we also chose $\theta=1$ in both subproblems. The damping parameters $w_k$ are set to $0.9$ to remain close to the undamped
version of the algorithm (convergence of the method for some conditions on $w_k$ have been proven in \cite{burger2008}).

The parameters $\alpha$, $\beta$ and $\delta$ were optimized by comparing the final results
with the existing ground truth in both cases. Here we mention that the choice of parameters is not a trivial task, since the result strongly varies with a change in the parameters.
In figure \ref{fig:ErrorParams}, the scaling between the error in the $L^2$-norm between exact and reconstructed image sequence per pixel per time step and the choice of each
parameter out of a certain range is displayed examplarily for the heart data set. Here, we chose $\alpha\in[0,0.5]$, $\beta\in[0,1]$ and $\delta\in[0,2]$ and kept two parameters
fixed while plotting the error in the third one. The adaption of parameters in case of real data and, if possible, the elimination of some of them remains a future task.

\begin{figure}[ht]
\centering
\subfloat[Error scaling for $\alpha$]{\includegraphics[width=5cm]{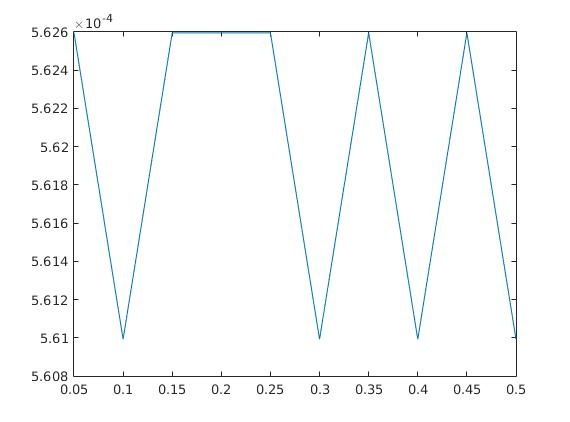}} \
\subfloat[Error scaling for $\beta$]{\includegraphics[width=5cm]{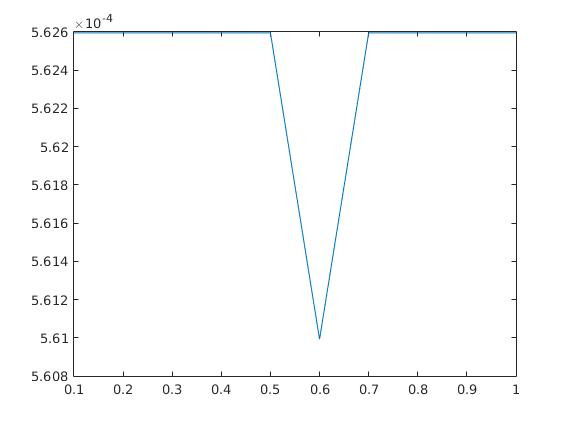}} \
\subfloat[Error scaling for $\delta$]{\includegraphics[width=5cm]{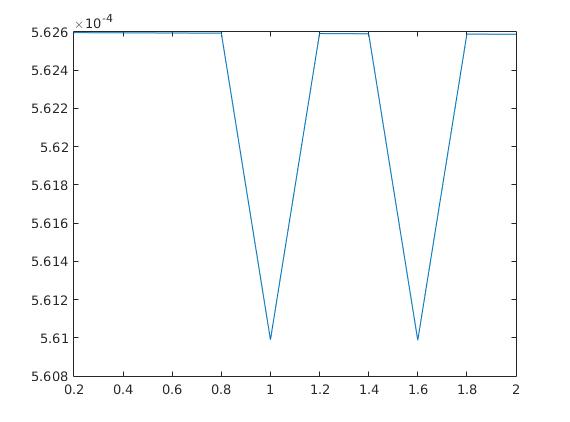}} \
\caption{$L^2$-error between exact and reconstructed data for different regularization parameters for the first data set (for given data without noise)}
\label{fig:ErrorParams}
\end{figure}

In a first test, every image sequence was reconstructed out of the exact given sinograms. Additionally we tested noise corrupted data by first scaling the sinogram by a parameter $\frac{1}{10^{11}}$, corrupting them with Poisson noise via the MATLAB imaging toolbox command \texttt{imnoise} and finally rescaling the image to the original range (see figure \ref{fig:NoisySinograms}). The average
count number per time step (i.e. the average of the discrete $\ell_1$-norm of the data at each time step) is approximately $2.5\cdot 10^3$ in case of the heart-shaped data set and
ca. $800$ in case of the rat liver simulation. The results at
a certain number of time steps can be seen in figure \ref{fig:ResultSet1} and \ref{fig:ResultSet2}. For comparison, we additionally performed a reconstruction with a simple
alternating EM method, keeping the assumption that the tracer can be modelled as a sum of indicator functions and subconcentration curves, but neglecting any regularization terms.
In all tests, the outer iteration number was set to 1000 with 10000 inner iterations per subproblem, to obtain a result within a reasonable time period. As stopping criterion, we
chose the primal dual residual (cp. \cite{Goldstein2013}) for the inner and the maximum over the Frobenius norms of $U$ and $C$ for the outer iterations. The results are displayed
in \ref{fig:ResultSet1} and \ref{fig:ResultSet2} respectively.

\begin{figure}[h!]
\captionsetup[subfigure]{labelformat=empty}
\centering
\subfloat[t=1]{\includegraphics[width=2cm]{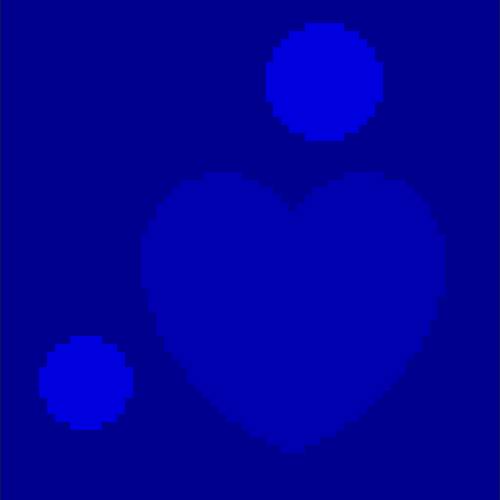}} \
\subfloat[t=5]{\includegraphics[width=2cm]{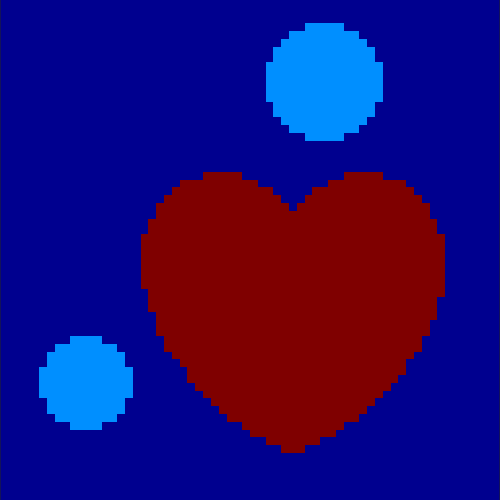}} \
\subfloat[t=10]{\includegraphics[width=2cm]{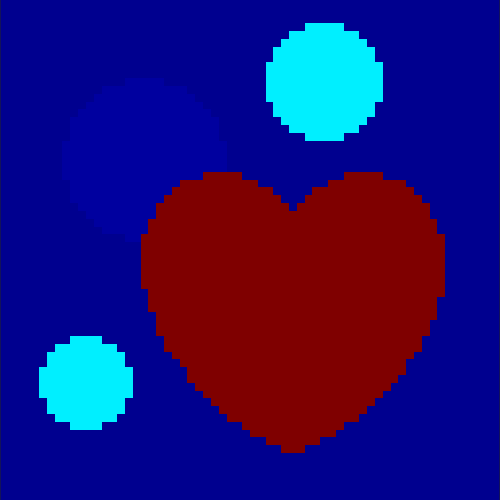}} \
\subfloat[t=15]{\includegraphics[width=2cm]{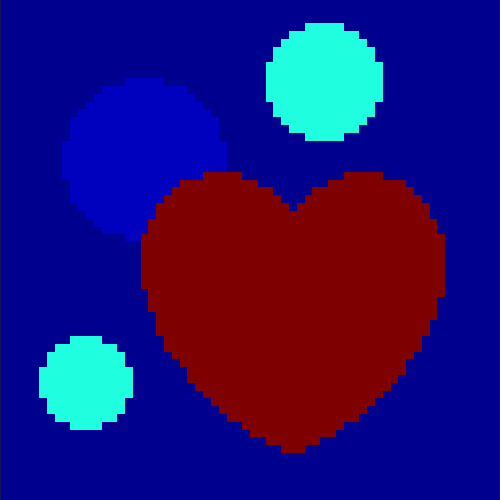}} \
\subfloat[t=25]{\includegraphics[width=2cm]{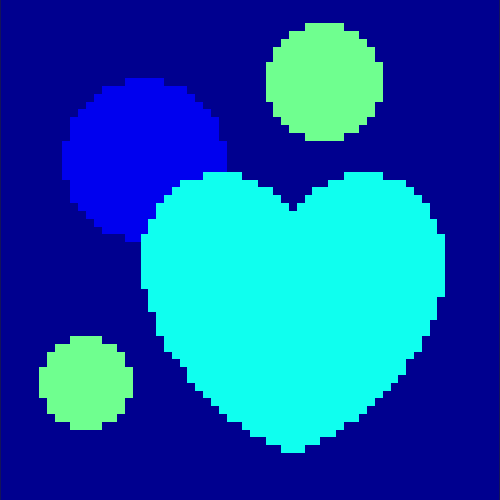}} \
\subfloat[t=50]{\includegraphics[width=2cm]{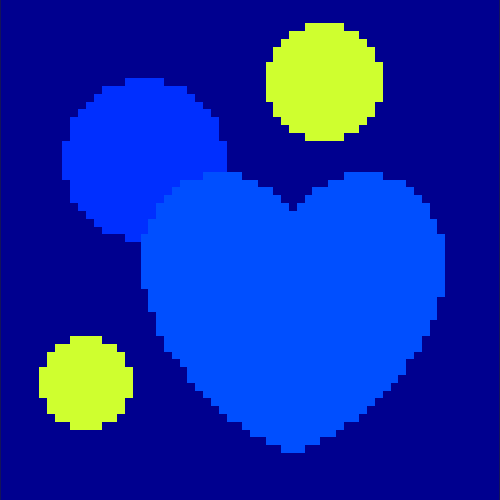}} \
\subfloat[t=90]{\includegraphics[width=2cm]{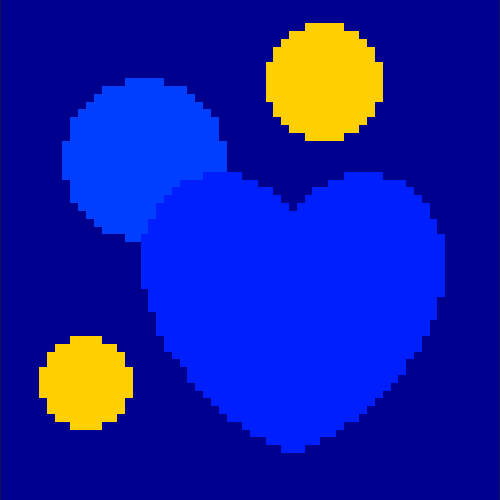}} \
\vspace{0.01cm}
\subfloat[t=1]{\includegraphics[width=2cm]{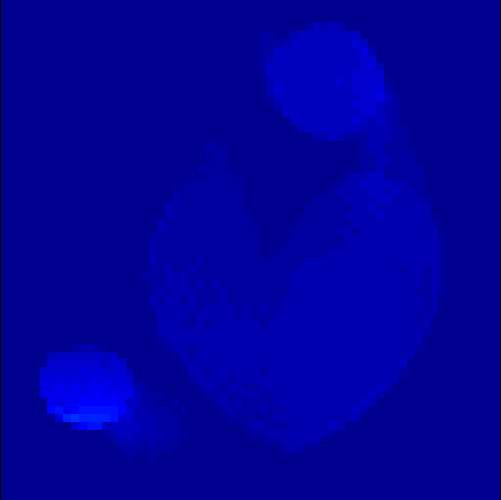}} \
\subfloat[t=5]{\includegraphics[width=2cm]{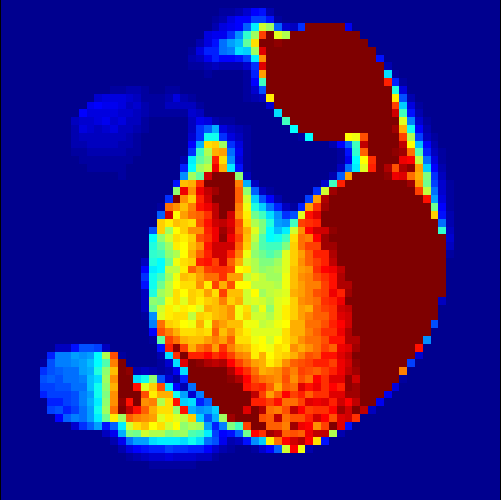}} \
\subfloat[t=10]{\includegraphics[width=2cm]{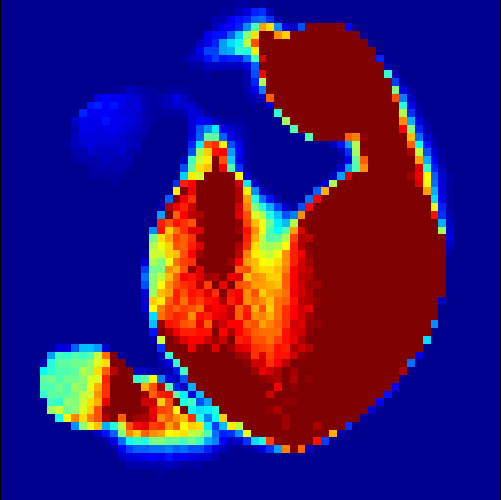}} \
\subfloat[t=15]{\includegraphics[width=2cm]{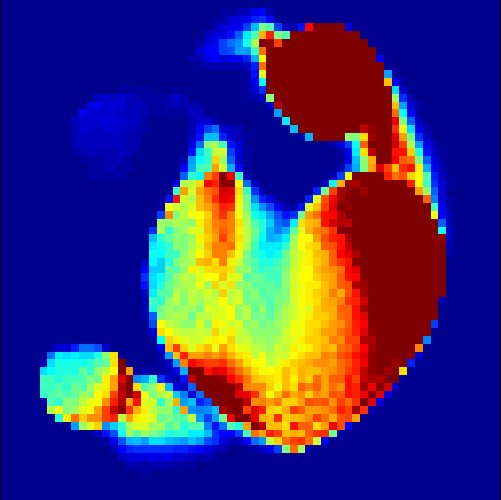}} \
\subfloat[t=25]{\includegraphics[width=2cm]{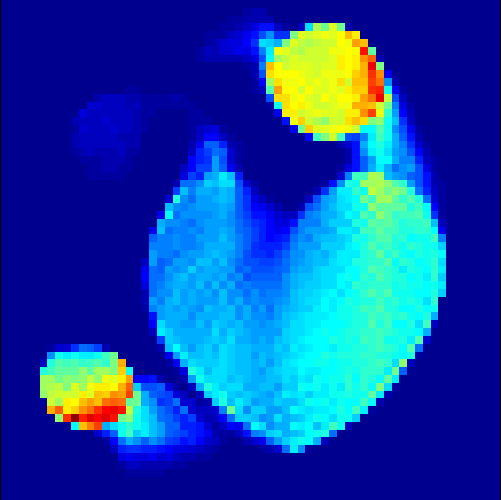}} \
\subfloat[t=50]{\includegraphics[width=2cm]{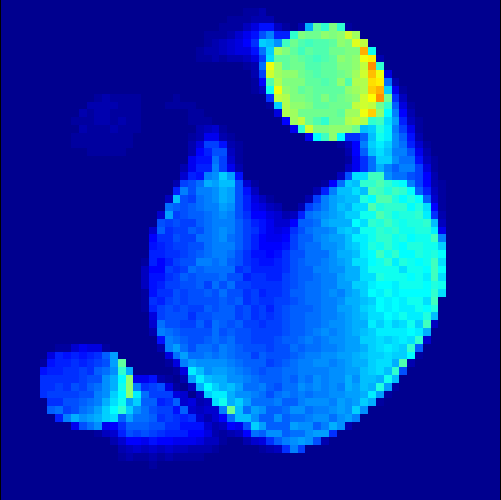}} \
\subfloat[t=90]{\includegraphics[width=2cm]{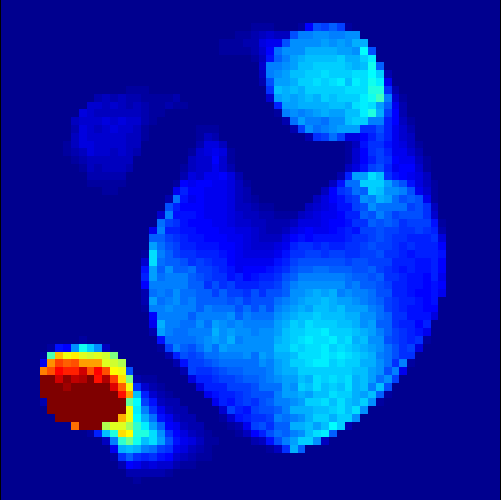}} \
\vspace{0.01cm}
\subfloat[t=1]{\includegraphics[width=2cm]{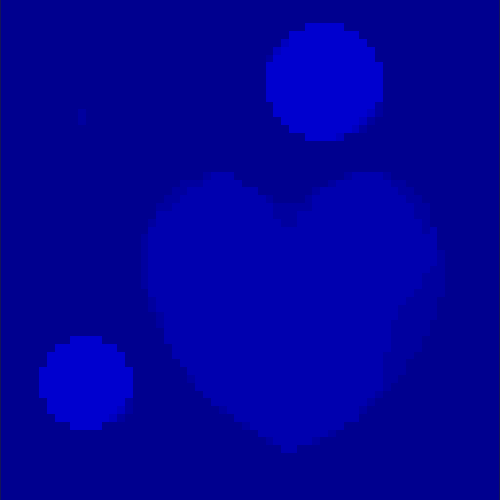}} \
\subfloat[t=5]{\includegraphics[width=2cm]{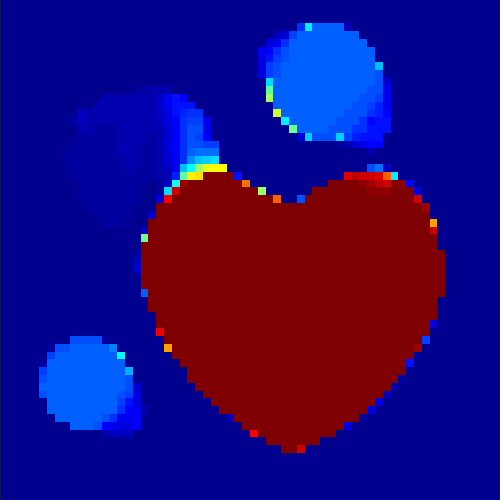}} \
\subfloat[t=10]{\includegraphics[width=2cm]{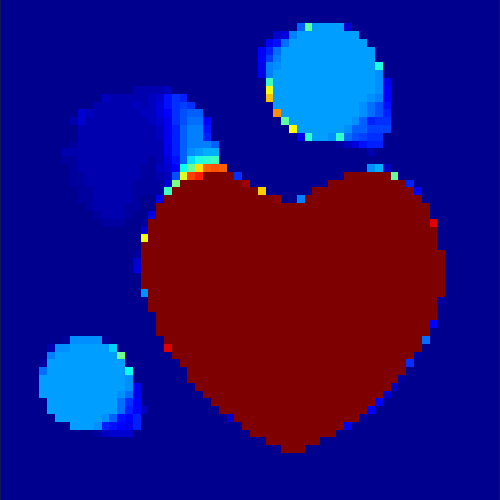}} \
\subfloat[t=15]{\includegraphics[width=2cm]{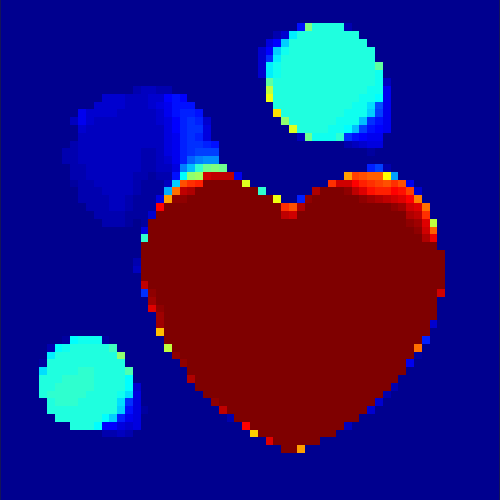}} \
\subfloat[t=25]{\includegraphics[width=2cm]{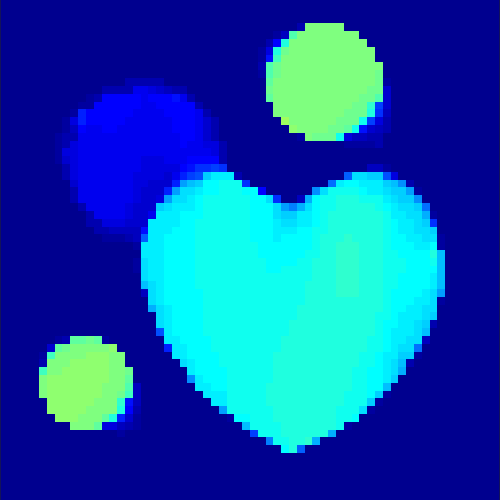}} \
\subfloat[t=50]{\includegraphics[width=2cm]{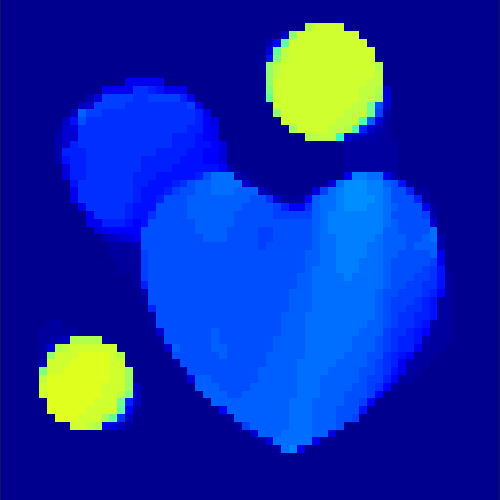}} \
\subfloat[t=90]{\includegraphics[width=2cm]{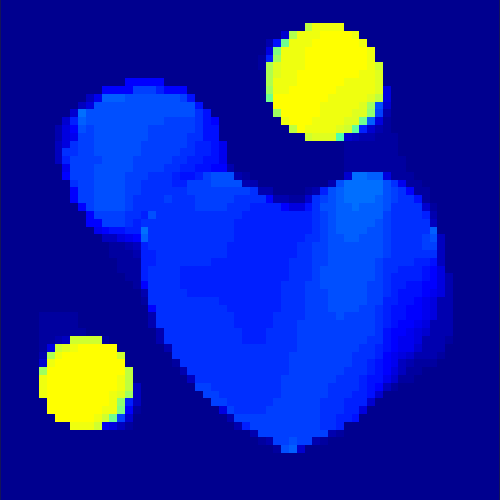}}
\vspace{0.01cm}
\subfloat[t=1]{\includegraphics[width=2cm]{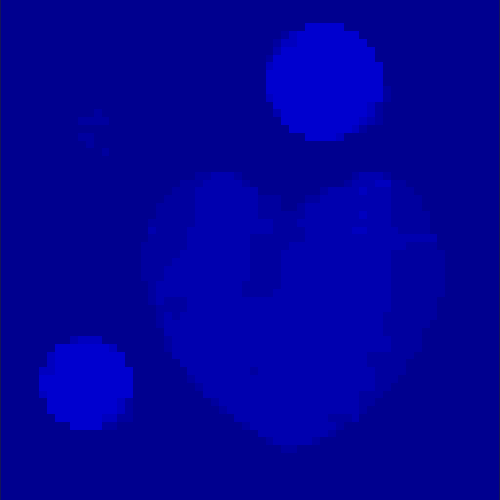}} \
\subfloat[t=5]{\includegraphics[width=2cm]{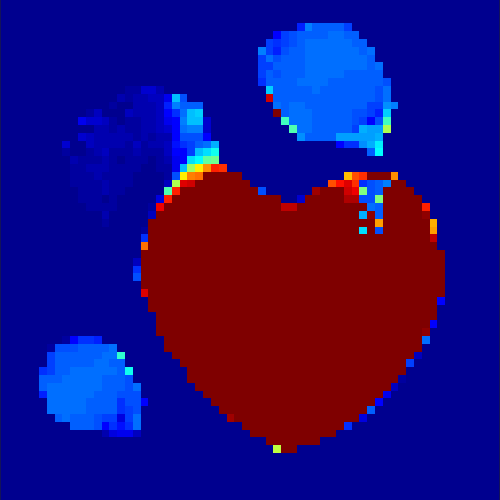}} \
\subfloat[t=10]{\includegraphics[width=2cm]{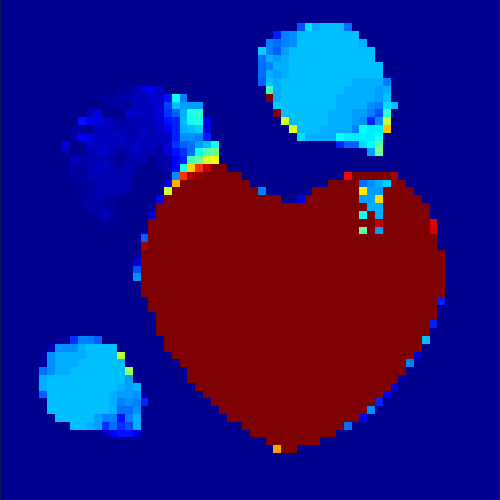}} \
\subfloat[t=15]{\includegraphics[width=2cm]{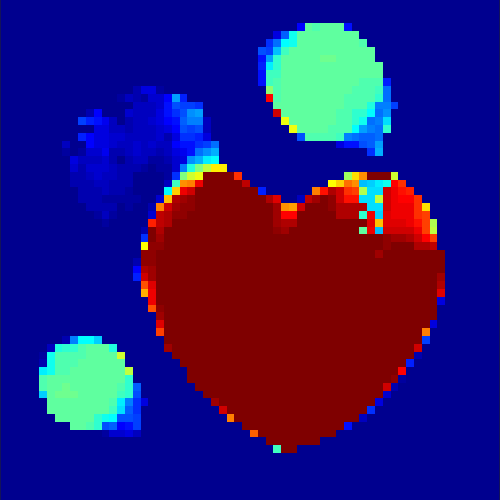}} \
\subfloat[t=25]{\includegraphics[width=2cm]{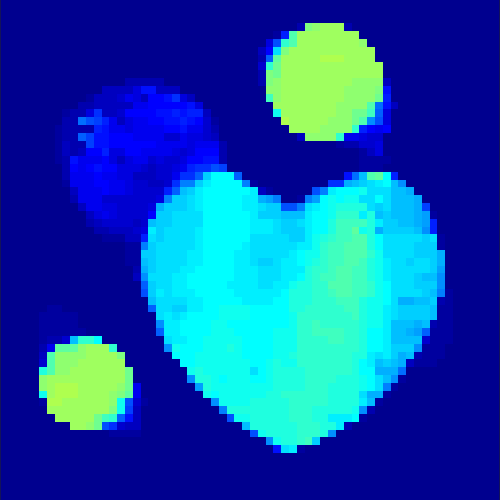}} \
\subfloat[t=50]{\includegraphics[width=2cm]{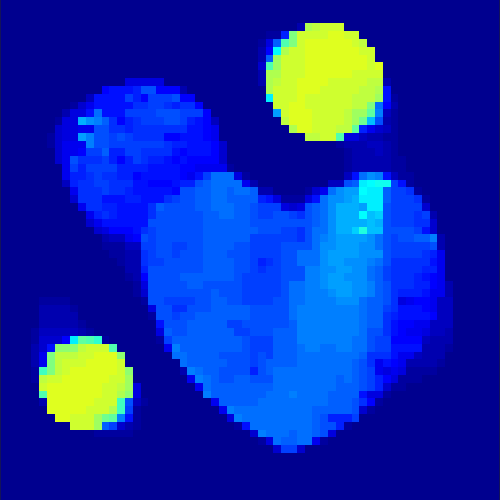}} \
\subfloat[t=90]{\includegraphics[width=2cm]{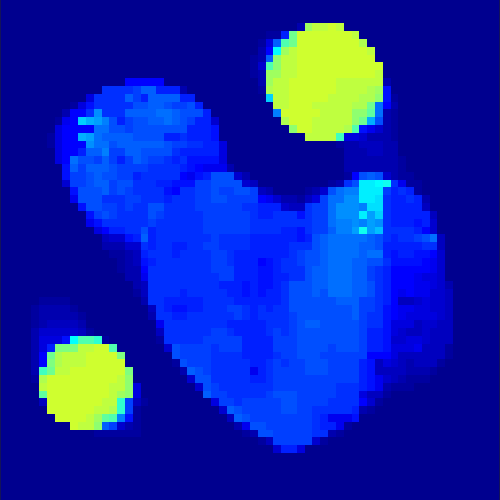}}
\caption{Reconstruction of simple regions: Exact image sequence (first row), reconstructed solution with simple alternating EM method without regularization (second row), reconstructed solution with exact given data (third row), reconstructed solution with Poisson
noise-corrupted data (fourth row)}
\label{fig:ResultSet1}
\end{figure}

\begin{figure}[h!]
\captionsetup[subfigure]{labelformat=empty}
\centering
\subfloat[t=1]{\includegraphics[width=2cm]{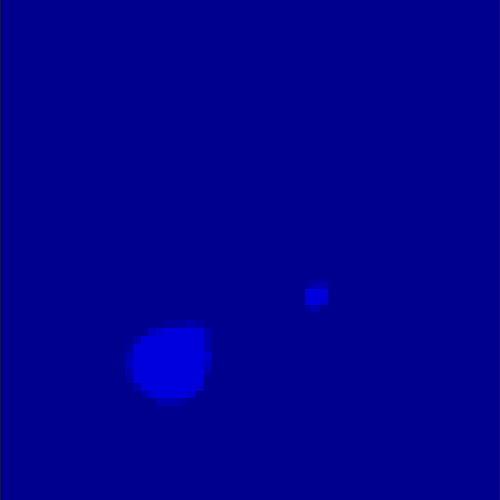}} \
\subfloat[t=5]{\includegraphics[width=2cm]{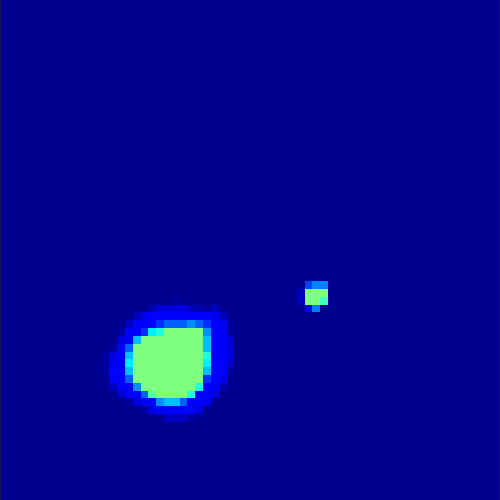}} \
\subfloat[t=10]{\includegraphics[width=2cm]{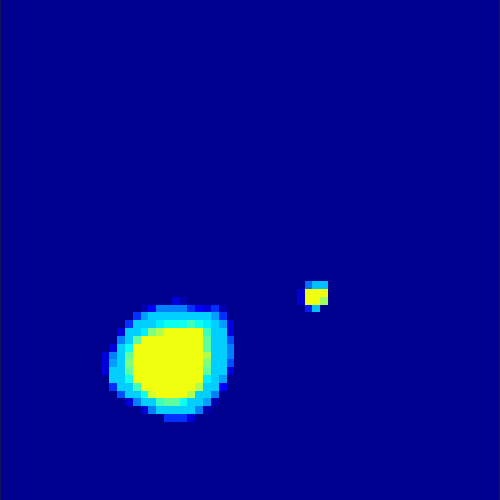}} \
\subfloat[t=15]{\includegraphics[width=2cm]{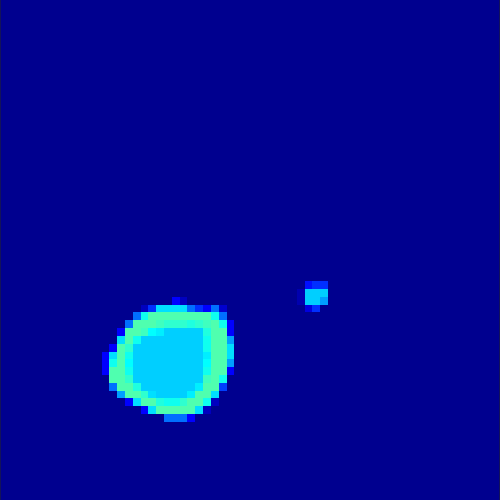}} \
\subfloat[t=25]{\includegraphics[width=2cm]{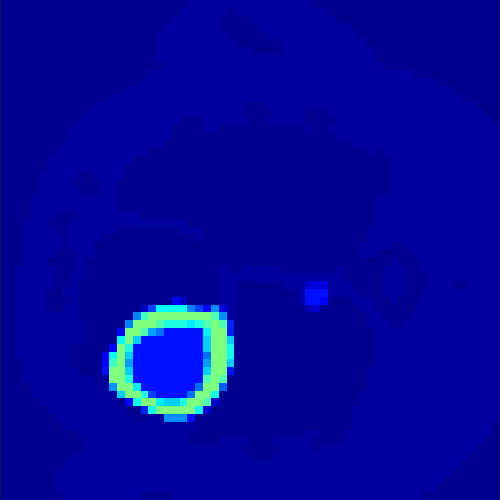}} \
\subfloat[t=50]{\includegraphics[width=2cm]{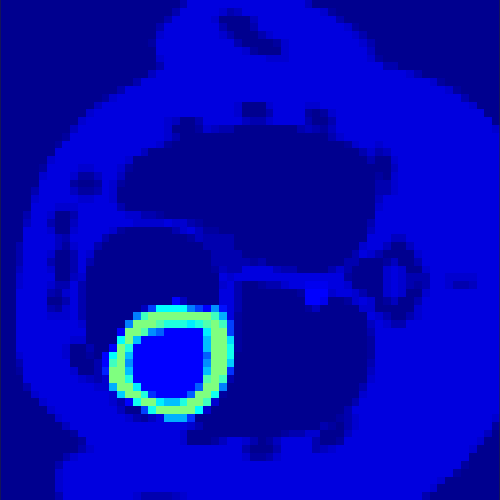}} \
\subfloat[t=90]{\includegraphics[width=2cm]{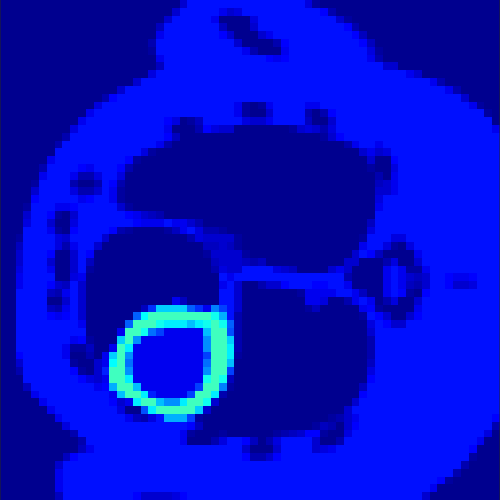}} \
\vspace{0.01cm}
\subfloat[t=1]{\includegraphics[width=2cm]{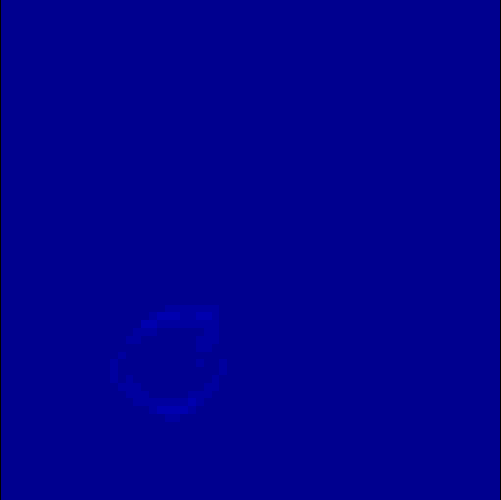}} \
\subfloat[t=5]{\includegraphics[width=2cm]{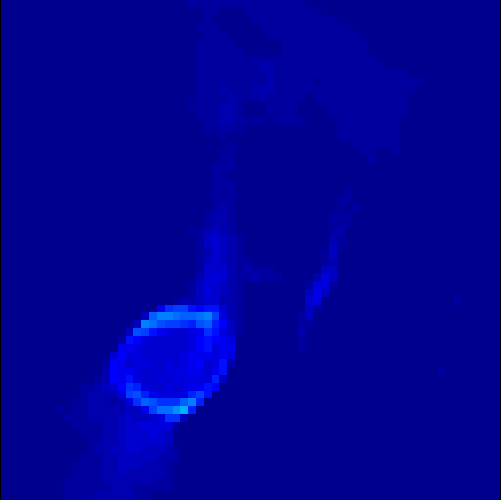}} \
\subfloat[t=10]{\includegraphics[width=2cm]{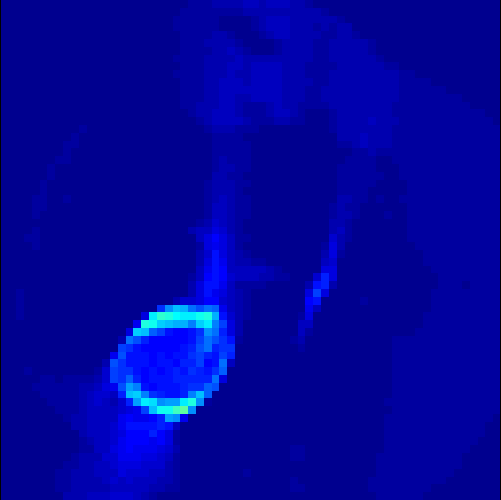}} \
\subfloat[t=15]{\includegraphics[width=2cm]{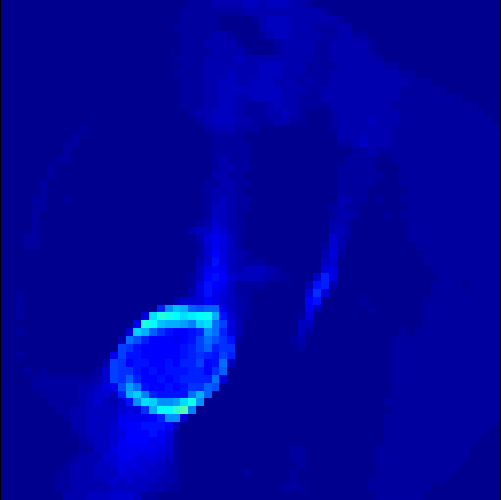}} \
\subfloat[t=25]{\includegraphics[width=2cm]{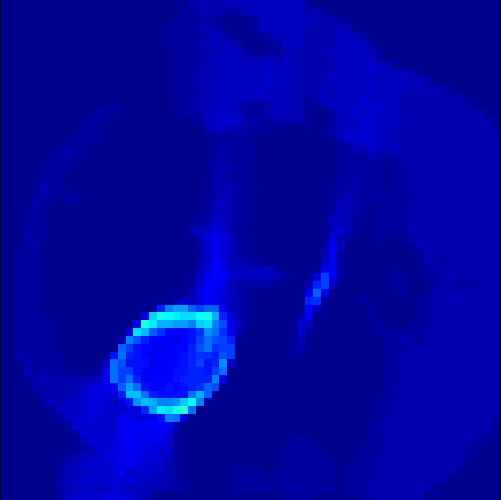}} \
\subfloat[t=50]{\includegraphics[width=2cm]{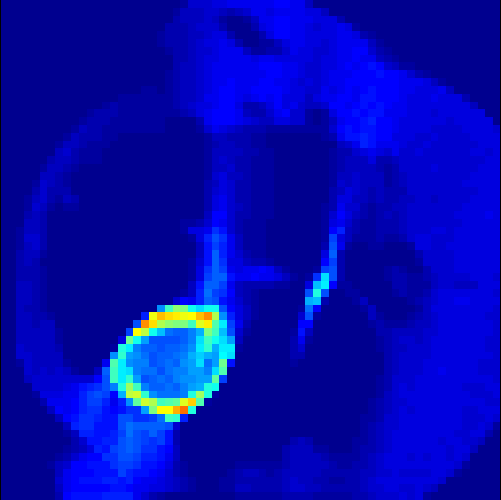}} \
\subfloat[t=90]{\includegraphics[width=2cm]{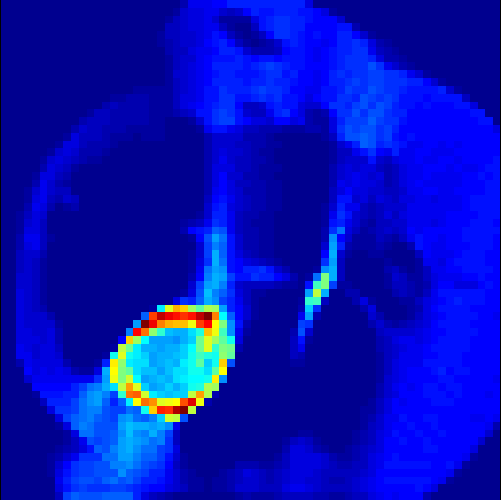}} \
\vspace{0.01cm}
\subfloat[t=1]{\includegraphics[width=2cm]{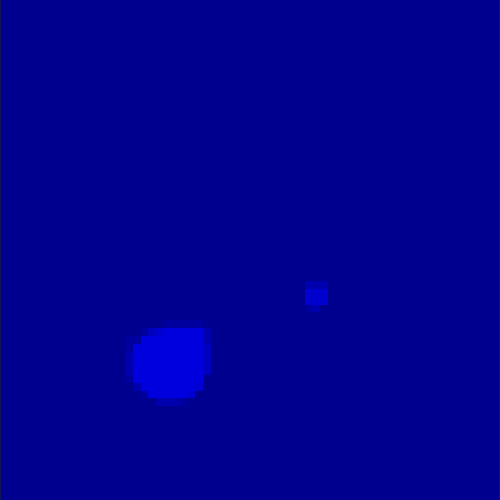}} \
\subfloat[t=5]{\includegraphics[width=2cm]{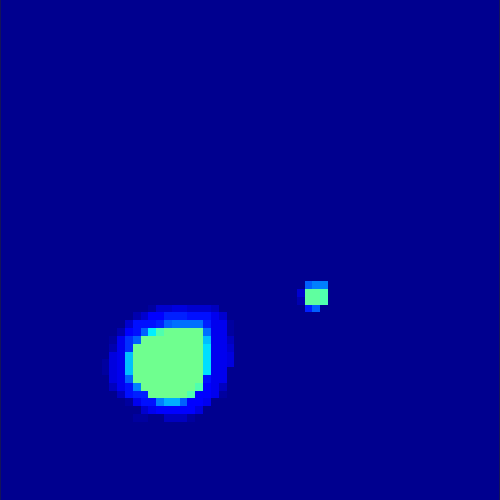}} \
\subfloat[t=10]{\includegraphics[width=2cm]{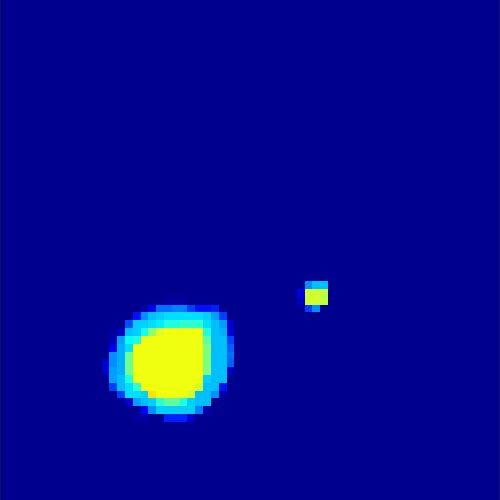}} \
\subfloat[t=15]{\includegraphics[width=2cm]{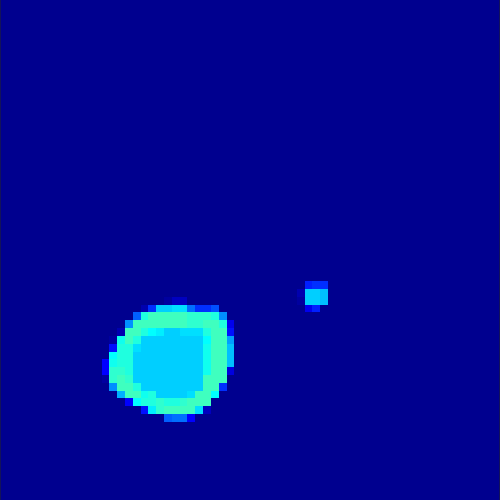}} \
\subfloat[t=25]{\includegraphics[width=2cm]{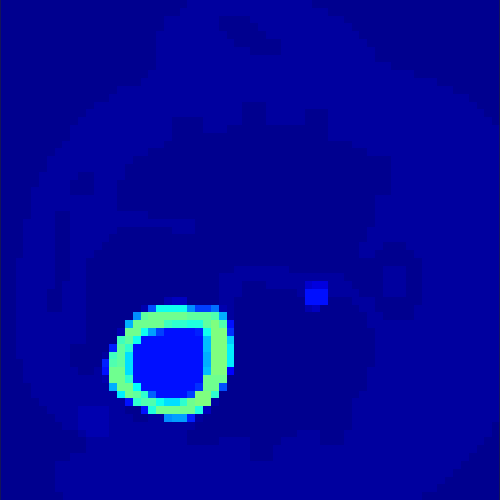}} \
\subfloat[t=50]{\includegraphics[width=2cm]{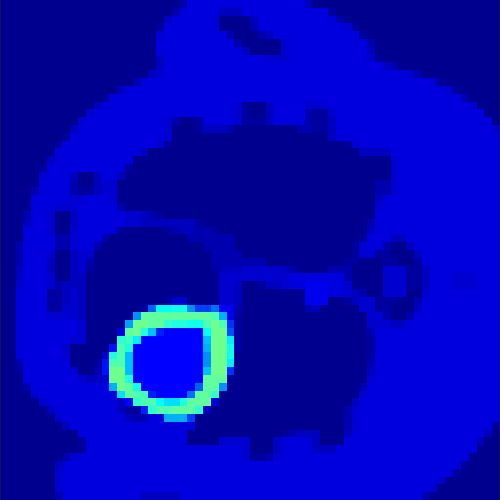}} \
\subfloat[t=90]{\includegraphics[width=2cm]{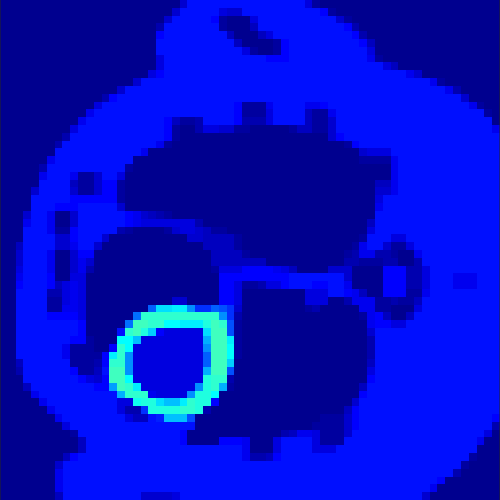}}
\vspace{0.01cm}
\subfloat[t=1]{\includegraphics[width=2cm]{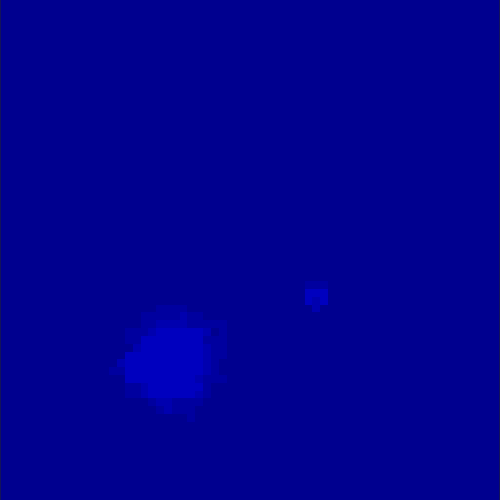}} \
\subfloat[t=5]{\includegraphics[width=2cm]{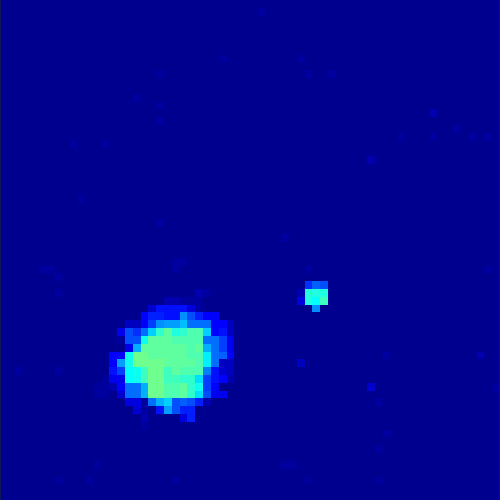}} \
\subfloat[t=10]{\includegraphics[width=2cm]{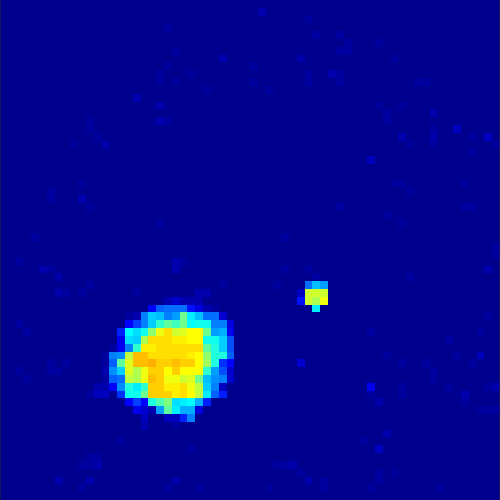}} \
\subfloat[t=15]{\includegraphics[width=2cm]{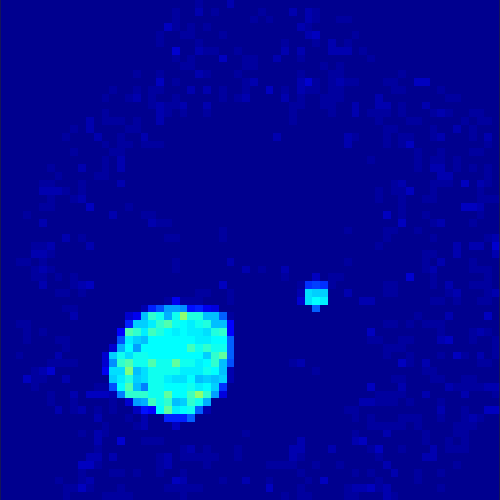}} \
\subfloat[t=25]{\includegraphics[width=2cm]{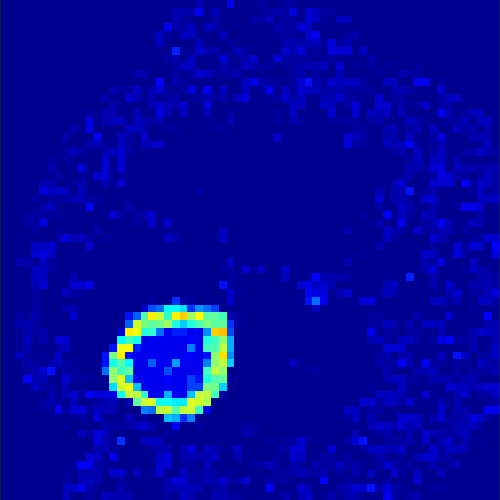}} \
\subfloat[t=50]{\includegraphics[width=2cm]{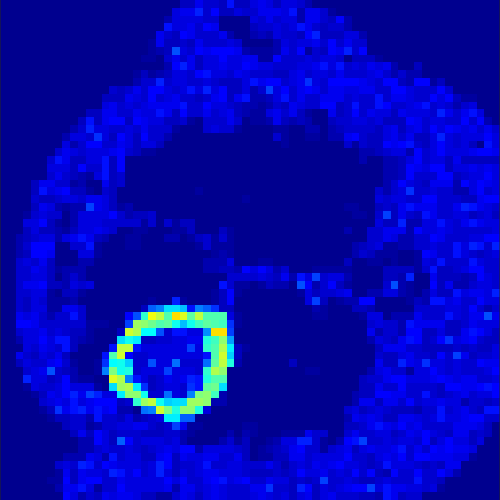}} \
\subfloat[t=90]{\includegraphics[width=2cm]{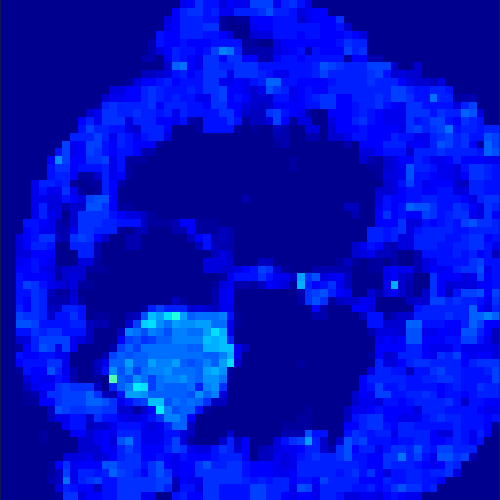}}
\caption{Reconstruction of synthesized rat liver: Exact image sequence (first row), reconstructed solution with simple alternating EM method without regularization (second row),
reconstructed solution with exact given data (third row), reconstructed solution with Poisson noise-corrupted data (fourth row)}
\label{fig:ResultSet2}
\end{figure}

As one can see in both figures, the reconstruction method applied to each data set performs very well, especially in contrast to the simple alternating EM method. This clearly shows
the benefits of the proposed regularization methods. In case of noise-free given data, the shape of every object, where especially
the heart is of higher interest, is clearly defined. As expected, we often observe errors in the edges of each region and where two regions are directly connected (the heart and
the upper left circle). This causes the algorithm to incorrectly assign these pixels to another region. Furthermore, the reconstruction difficulties increase with an increase in noise.
Some more pixels are assigned to the wrong region, which leads to a small hole-like structure within the heart region and causes a slight blurring effect. In the second data set the
method clearly outperforms several other approaches by providing very clearly defined regions and even reconstructing fine structures of the phantom. However, as mentioned before,
a clear reconstruction of the rat liver required highly optimized parameter sets, which makes the whole problem quite susceptible to parameter changes.

All in all, the results presented in this work offer a novel point of view to the topic and provide a flexible reconstruction approach which allows some room for improvement until it
is suitable for being applied to real clinical data (cp. section \ref{sec:Conclusion}).


\subsection{Monte Carlo Simulation}

In order to test the behaviour of the proposed method in a more realistic, random-based test case, we performed a Monte Carlo simulation for dynamic SPECT imaging. First, we created
a simple $129\times 129$ image phantom consisting of an outer and two inner circles which represents the structure of the region of interest (see figure \ref{fig:MonteCarloData}(a)). Within those regions we
assumed concentration curves over a time period of 90 time steps as displayed in figure \ref{fig:MonteCarloData}(b). Based on the tracer intensity in an image frame at each time step, we created a variable
number of random decay events (where the number is proportional to the average concentration in one pixel in the whole image frame per time step) with a probability proportional to the concentration in every
subregion. They are detected by a virtual double head gamma camera rotating around the patient by 46 degrees per time step, which consists of 374 detector bins. Every simulated decay event is projected onto the
scanner and counted by the corresponding detector bin.

\begin{figure}[ht]
\centering
\subfloat[Monte Carlo image phantom]{\includegraphics[width=7cm]{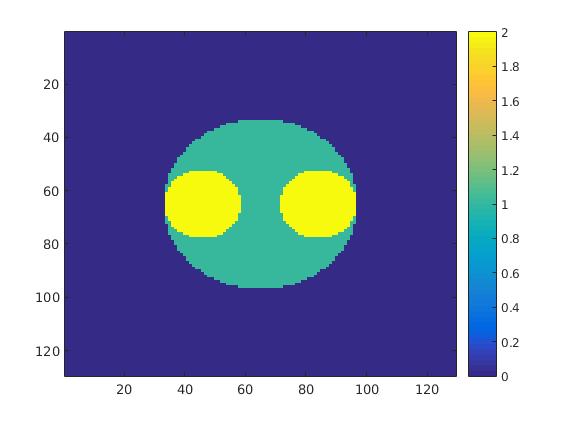}} \
\subfloat[Simulated concentration curves]{\includegraphics[width=7cm]{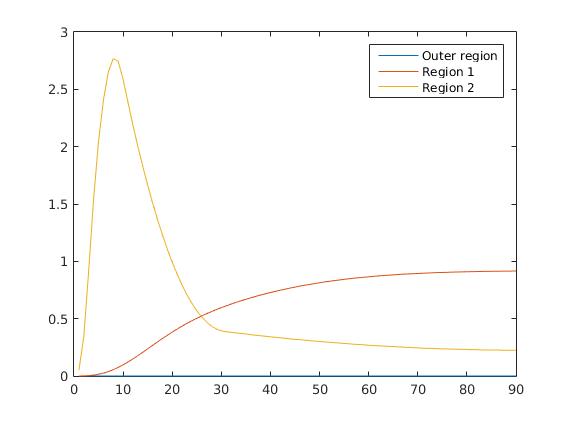}} \
\caption{Monte Carlo simulation. Region 1 corresponds to the outer circle, region 2 to the two inner circles.}
\label{fig:MonteCarloData}
\end{figure}

In two different tests we fixed the number of events counted by the detector equal to $\lambda=20000$ (resp. $\lambda=200000$) times the average concentration in one pixel. The resulting sinogram images of the accumulated counts
in each bin are shown in figure \ref{fig:MonteCarloSinograms}.

\begin{figure}[ht]
\centering
\subfloat[Sinogram for $\lambda=20000$]{\includegraphics[width=7cm]{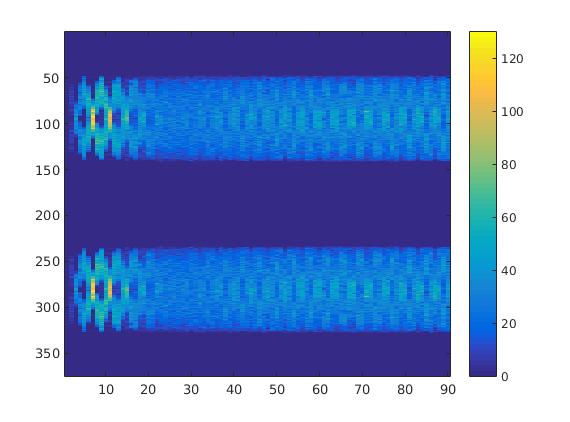}} \
\subfloat[Sinogram for $\lambda=200000$]{\includegraphics[width=7cm]{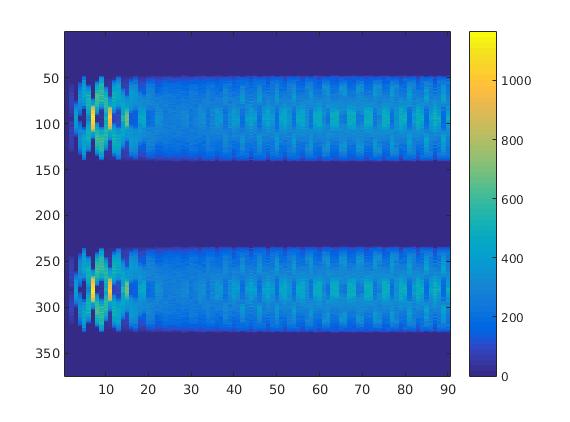}} \
\caption{Monte Carlo sinogram data}
\label{fig:MonteCarloSinograms}
\end{figure}

Based on the sinogram data we applied the proposed algorithm in order to reconstruct the original image sequence. The results for both test cases are shown in figure \ref{fig:MonteCarloResults}.

\begin{figure}[h!]
\captionsetup[subfigure]{labelformat=empty}
\centering
\subfloat[t=1]{\includegraphics[width=2cm]{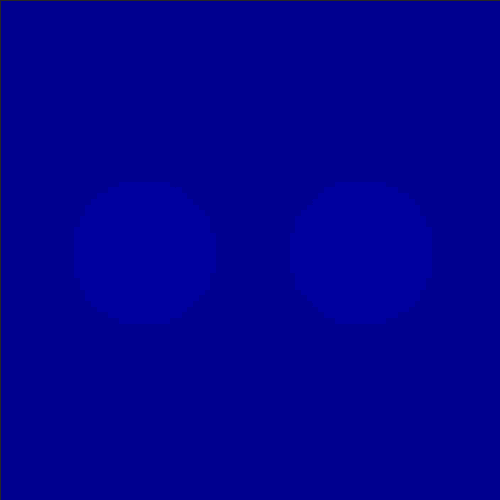}} \
\subfloat[t=5]{\includegraphics[width=2cm]{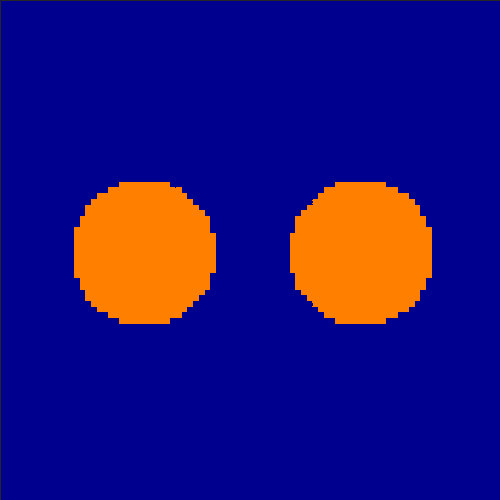}} \
\subfloat[t=10]{\includegraphics[width=2cm]{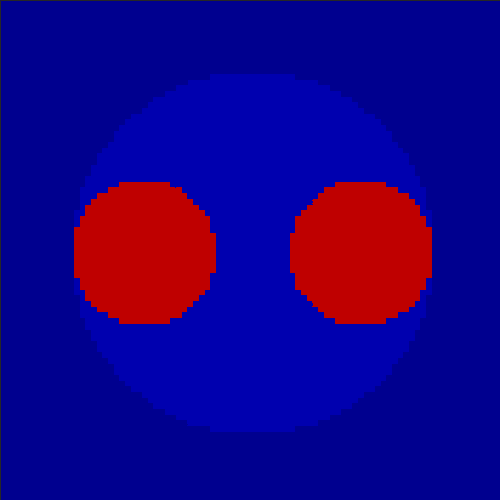}} \
\subfloat[t=15]{\includegraphics[width=2cm]{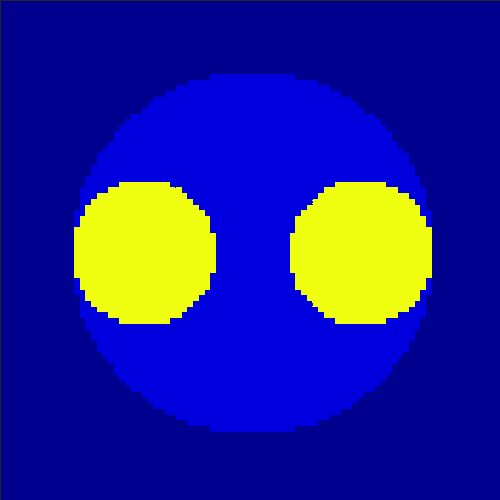}} \
\subfloat[t=25]{\includegraphics[width=2cm]{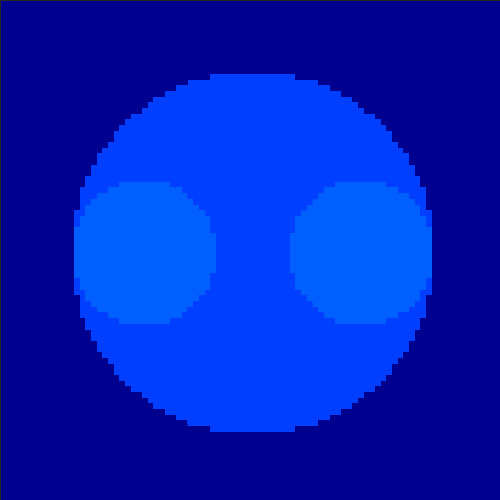}} \
\subfloat[t=50]{\includegraphics[width=2cm]{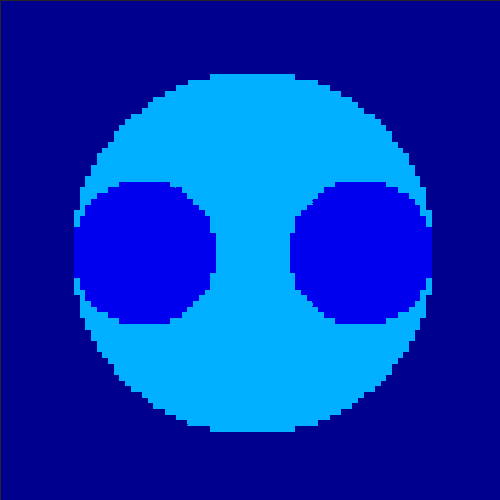}} \
\subfloat[t=90]{\includegraphics[width=2cm]{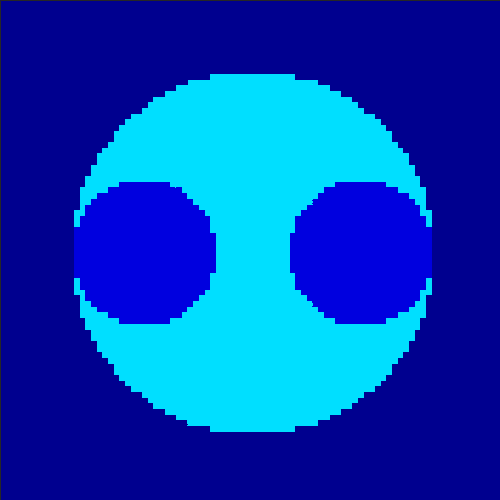}} \
\vspace{0.01cm}
\subfloat[t=1]{\includegraphics[width=2cm]{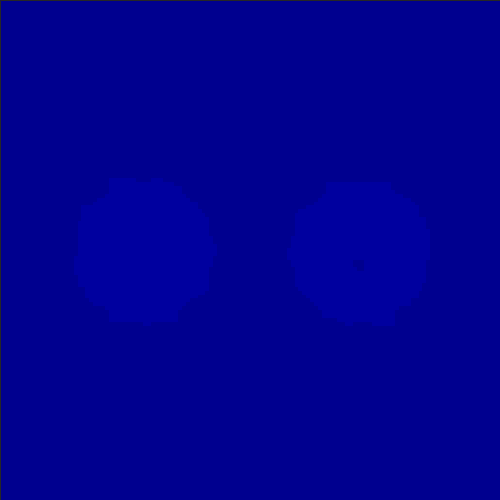}} \
\subfloat[t=5]{\includegraphics[width=2cm]{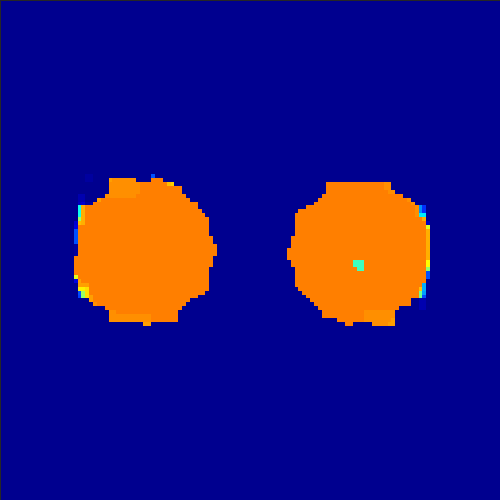}} \
\subfloat[t=10]{\includegraphics[width=2cm]{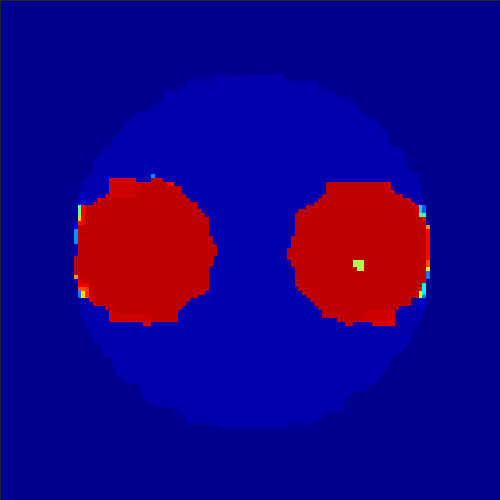}} \
\subfloat[t=15]{\includegraphics[width=2cm]{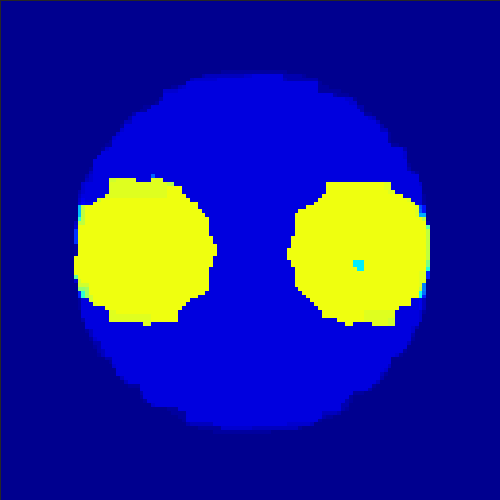}} \
\subfloat[t=25]{\includegraphics[width=2cm]{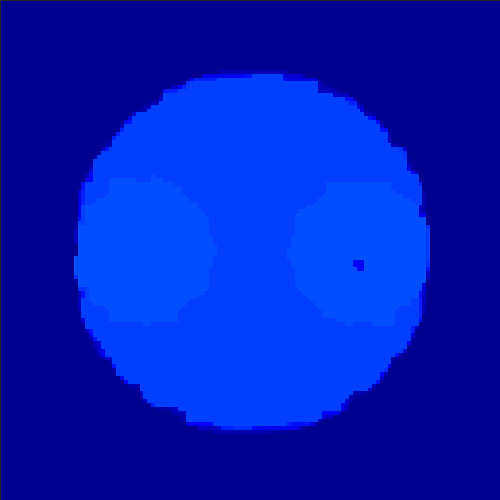}} \
\subfloat[t=50]{\includegraphics[width=2cm]{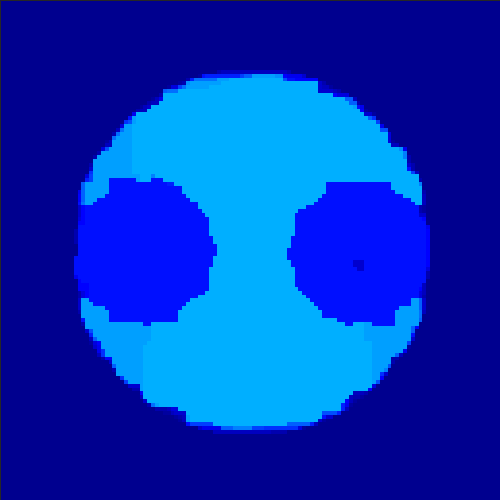}} \
\subfloat[t=90]{\includegraphics[width=2cm]{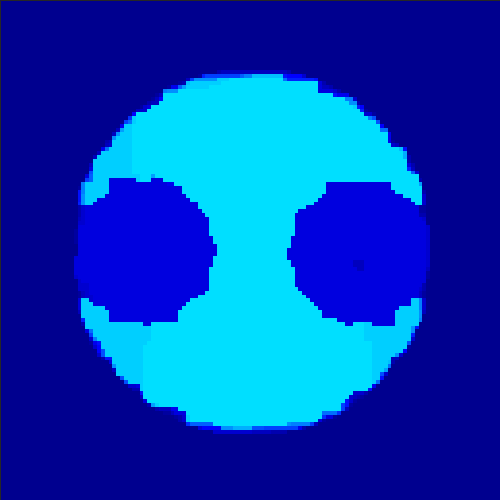}} \
\vspace{0.01cm}
\subfloat[t=1]{\includegraphics[width=2cm]{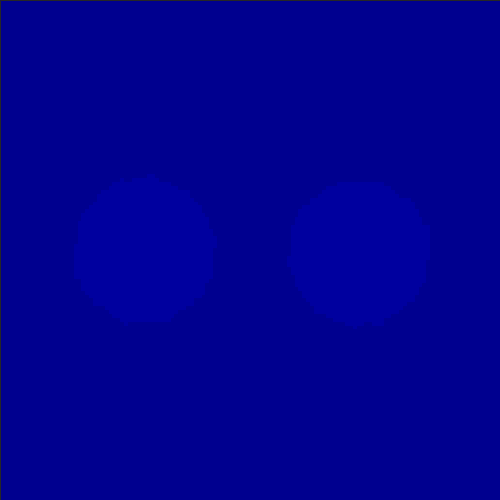}} \
\subfloat[t=5]{\includegraphics[width=2cm]{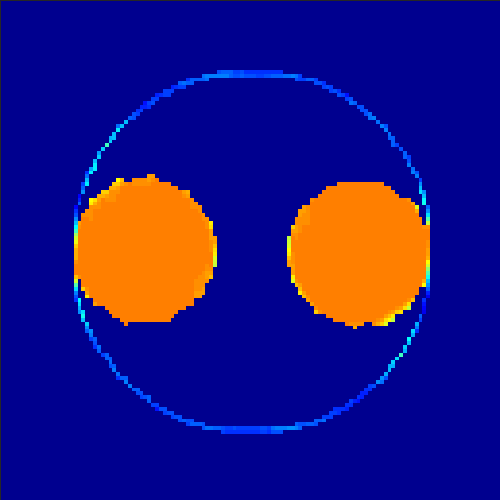}} \
\subfloat[t=10]{\includegraphics[width=2cm]{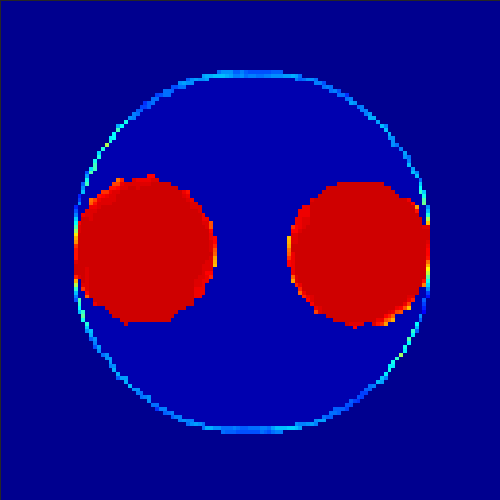}} \
\subfloat[t=15]{\includegraphics[width=2cm]{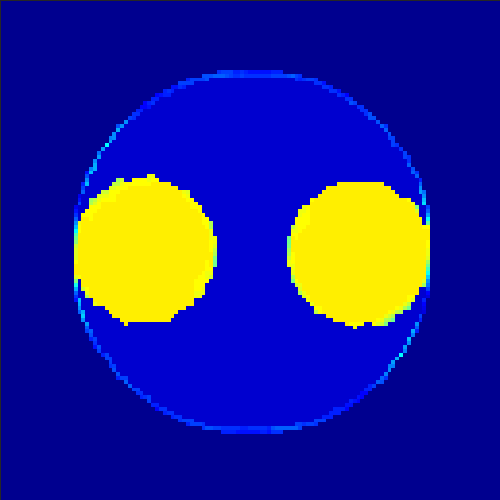}} \
\subfloat[t=25]{\includegraphics[width=2cm]{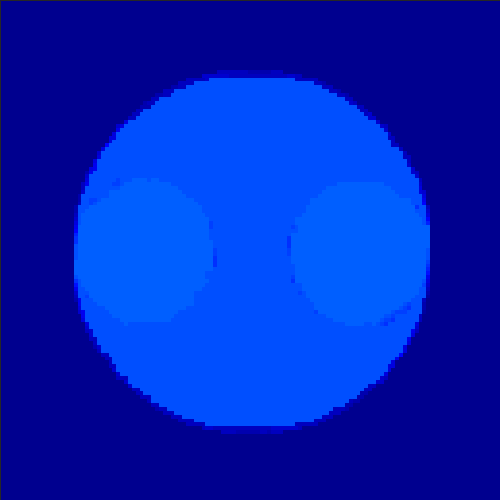}} \
\subfloat[t=50]{\includegraphics[width=2cm]{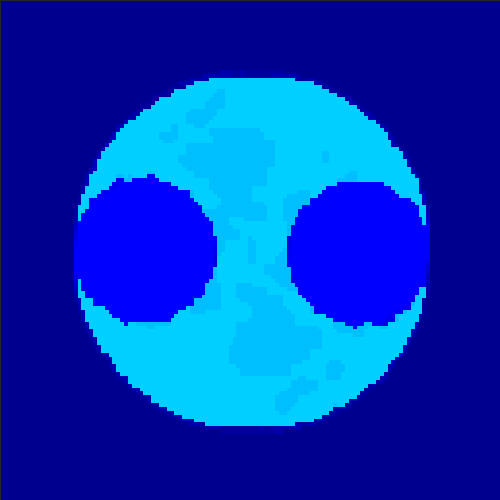}} \
\subfloat[t=90]{\includegraphics[width=2cm]{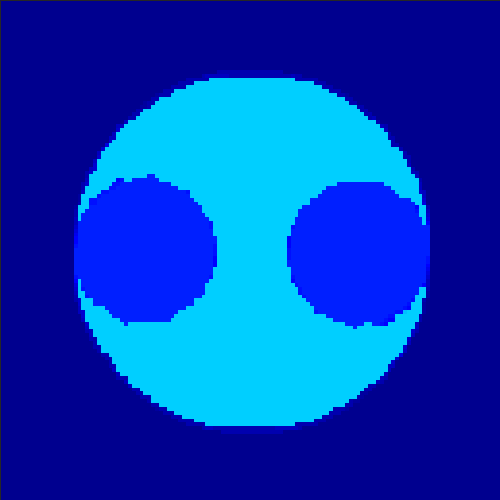}}
\caption{Reconstruction of Monte Carlo simulated data: Exact image sequence (first row), reconstructed solution with $\lambda=20000$ (second row),
reconstructed solution with $\lambda=200000$ (third row)}
\label{fig:MonteCarloResults}
\end{figure}

As one can see, the method is able to reconstruct the regions properly, even in case of a low count number. Within a number of iterations (average of 100 outer and 10000 inner iterations), the
algorithm presents a reasonable reconstruction of the region of interest and the corresponding regional tracer concentration curves. Here, the parameters were not optimized as in the case of the synthesized data
sets in the previous section, but kept fixed as $\alpha=0.5$, $\beta=0.1$ and $\delta=1.5$. With futher optimized parameter values one could possibly provide even better results.

All in all, the proposed method provides promising results that could be able to be applied to several studies in biomedical imaging.


\section{Conclusion and Outlook}
\label{sec:Conclusion}

In this paper we presented a new simultaneous reconstruction approach in dynamic SPECT imaging, derived and implemented a suitable variational model and presented promising results
on synthesized data. The proposed method yields
promising results, which we have demonstrated on both exact and noisy  data. The results were quite plausible and the reconstructed image sequences significantly match the exact ones. The reconstructions prove that the choice of the regularization methods as well as the reconstruction approach is reasonable for a proof-of-concept study. However, some limitations and open questions
arising from this work that could be addressed in future work are the following:

\begin{itemize}

\item Although our method provides promising results for the different data sets, we think that a direct comparison to other existing approaches like the ones mentioned in the first sections may not be significant as the current model does not take into account the motion on the boundary, and those methods use a dictionary of basis functions for the tracer concentration curves or make additional assumptions on the structure of the region of interest. Such a comparison in case of real data would thus be the next step in order to further promote this novel approach.

\item The proposed variational formulation contains at least, as in the constrained version, three regularization parameters. Furthermore, some of the proposed algorithms require additional proximal parameters, which have to be chosen with respect to certain properties of the data and regularization functionals. A future task could be to discuss the parameter choice in detail and maybe to improve the model by eliminating some of them by setting some parameter proportional to another, for instance (e.g. fixing a ratio between $BV$- and $L^1$-regularization).

\item In order to make the approach applicable to real data, one also has to face the problem of extending the idea to three-dimensional real image sequences. Therefore, this would
automatically lead to a computational problem that we may need to address in the future.

\end{itemize}


\section*{Acknowledgements}

This work has been initiated during a stay of CR at Shanghai funded by the Heinrich Hertz Foundation and the China Scholarship Council, whose support is gratefully acknowledged. MB
acknowledges further support by the German Science Foundation DFG through grant BU 2327/6-1 as well as SFB 656 Molecular Cardiovascular Imaging, and by ERC via Grant EU FP 7 - ERC
Consolidator Grant 615216 LifeInverse. The work of XZ is partially supported by NSFC91330102 and Sino-German Center grant (GZ1025), and 973 program (\# 2015CB856004). The authors thank Qiu Huang from Shanghai Jiao Tong University for providing the test rat liver phantom, Frank W\"ubbeling from the Westf\"alische Wilhelms-Universit\"at M\"unster for creating the Monte Carlo
simulation and the anonymous referees for their useful advice.

\section*{References}



\end{document}